\documentclass[oneside,12pt]{amsart}

\usepackage{amscd,amsmath,latexsym,bm}
\usepackage[mathcal]{euscript}
\usepackage[T1]{fontenc}
\usepackage{lmodern}
\usepackage{amscd,amsmath,latexsym,bm,pdfpages}
\usepackage[mathcal]{euscript}
\usepackage{graphicx}
\usepackage{amssymb}          
\usepackage{epstopdf}
\usepackage[all]{xy}

\newtheorem{thm}{Theorem}[section]
\newtheorem{cor}[thm]{Corollary}
\newtheorem{lem}[thm]{Lemma}
\newtheorem{prop}[thm]{Proposition}

\theoremstyle{definition}
\newtheorem{defin}[thm]{Definition}
\theoremstyle{definition}
\newtheorem{exm}[thm]{Example}

\newtheorem{remark}[thm]{Remark}
\newtheorem{examples}[thm]{Examples}

\theoremstyle{remark}
\newtheorem*{rem}{Remark}
\newtheorem*{ack}{Acknowledgment}

\newcommand{\R}{{\mathbb R}}
\newcommand{\C}{{\mathbb C}}

\begin{document}
\def\X#1#2{r(v^{#2}\ds{\prod_{i \in #1}}{x_{i}})}
\def\skp#1{\vskip#1cm\relax}
\def\C{{\mathbb C}}
\def\R{{\mathbb R}}
\def\ce{{\mathbb C}}
\def\erre{{\mathbb R}}
\def\efe{{\mathbb F}}
\def\ene{{\mathbb N}}
\def\UNO{1\mkern-7mu1}
\def\ee{{\mathbb E}}
\def\pee{{\mathbb P}}
\def\ene{{\mathbb N}}
\def\tn{{\bf T$^{n}$}}
\def\pn{$P^{n}\/$}
\def\cn{{\bf C$^{n}$}}
\def\z2{{\bf Z}$_{2}$}
\def\zl2{{\bf Z}$_{(2)}$}
\def\block{\rule{2.4mm}{2.4mm}}
\def\rplus{{\bf R$_{+}}}
\def\nd{\noindent}
\def\ds{\displaystyle}
\def\s{\sigma}
\def\sq{\operatorname{Sq}}
\numberwithin{equation}{section}
\title{The $\bm{KO^{*}}$-rings of $\bm{BT^m}$, the Davis-Januszkiewicz Spaces and
certain toric manifolds}

\author[L.~Astey]{L.~Astey}
\address{Department of Mathematics,
Cinvestav, San Pedro Zacatenco, Mexico, D.F. CP 07360 Apartado
Postal 14-740, Mexico}
\email{astey@math.cinvestav.mx}

\author[A.~Bahri]{A.~Bahri}
\address{Department of Mathematics,
Rider University, Lawrenceville, NJ 08648, U.S.A.}
\email{bahri@rider.edu}

\author[M.~Bendersky]{M.~Bendersky}
\address{Department of Mathematics
CUNY,  East 695 Park Avenue New York, NY 10065, U.S.A.}
\email{mbenders@xena.hunter.cuny.edu}

\author[F.~R.~Cohen]{F.~R.~Cohen}
\address{Department of Mathematics,
University of Rochester, Rochester, NY 14625, U.S.A.}
\email{cohf@math.rochester.edu}

\author[D.~M.~Davis]{D.~M.~Davis}
\address{Department of Mathematics,
Lehigh University, Bethlehem, PA 18015, U.S.A.}
\email{dmd1@lehigh.edu}

\author[S.~Gitler]{S.~Gitler}
\address{Department of Mathematics,
Cinvestav, San Pedro Zacatenco, Mexico, D.F. CP 07360 Apartado
Postal 14-740, Mexico} \email{sgitler@math.cinvestav.mx}

\author[M.~Mahowald]{M.~Mahowald}
\address{Department of Mathematics,
Northwestern University, Northwestern University
2033 Sheridan Road Evanston, IL 60208-2730, USA} 
\email{mark@math.northwestern.edu}

\author[N.~Ray]{N.~Ray}
\address{School of Mathematics,
University of Manchester, Oxford Road Manchester M13 9PL, United Kingdom} \email{nige@maths.manchester.ac.uk}

\author[R.~Wood]{R.~Wood}
\address{School of Mathematics,
University of Manchester, Oxford Road Manchester M13 9PL, United Kingdom} \email{Reg.Wood@manchester.ac.uk}



\begin{abstract}

This paper  contains an explicit computation of the 
$KO^{*}$-ring structure of an $m$-fold product of $\mathbb{C}P^{\infty}$, 
the Davis-Januszkiewicz spaces and of toric manifolds which have trivial $\sq^2$-homology. 
A key ingredient is the stable splitting of the Davis-Januszkiewicz spaces given in \cite{bbcg}.

\end{abstract}

\maketitle
\markboth{The $KO^{*}$-rings of $BT^m$, the Davis-Januszkiewicz Spaces and
certain toric manifolds}
{The $KO^{*}$-rings of $BT^m$, the Davis-Januszkiewicz Spaces and
certain toric manifolds}

\section{Introduction}\label{sec:introduction}
The term ``toric manifold'' in this paper refers to the topological space about which
detailed information may be found in
\cite{davis.jan} and  \cite{buchstaber.panov.2}; a brief description
is given in Section \ref{sec:tm}. These spaces are called also ``quasitoric manifolds''
and include the class of all non-singular projective toric varieties.

An $n$-torus $T^n$ acts on a toric manifold $M^{2n}$ with quotient space a simple polytope 
$P^n$ having $m$ codimension-one faces (facets). Associated to $P^n$ is a simplicial
complex $K_P$ on vertices $\{v_1,v_2, \ldots, v_m\}$ with each $v_i$ corresponding to a
single facet $F_i$ of $P^n$. The set  $\{v_{i_1},v_{i_2}, \ldots, v_{i_k}\}$ is a simplex
in $K_P$ if and only if 
$F_{i_1} \cap F_{i_2} \cap \ldots \cap F_{i_k} \neq \emptyset$.

The classifying space of the real $n$-torus $T^n$ is denoted by $BT^n$. 
Associated to the torus action is a Borel-space fibration

\begin{equation}\label{eqn:basic.fibration}
M^{2n} \longrightarrow  ET^n \times_{T^n} M^{2n}  
 \stackrel{p}{\longrightarrow}\ BT^n .\end{equation}

\

\nd Of course here, $ BT^n = 
\mathbb{C}P^{\infty} \times \mathbb{C}P^{\infty} \times \ldots \times \mathbb{C}P^{\infty},
\; (n \text{ factors})$.

\

The homotopy type of the Borel space $ET^n \times_{T^n} M^{2n}$ depends
on $K_P$ only. It is referred to as the Davis-Januszkiewicz  space of $K_P$ and is 
denoted by the symbol $\mathcal D\mathcal J(K_P)$. More generally, a Davis-Januszkiewicz
space exists for {\em any\/} simplicial complex $K$; Section \ref{sec:dj} contains
more details about this generalization. It is  known that  for any
complex-oriented cohomology theory $E^*$

\begin{equation}\label{eqn:dj.definition}
E^*(\mathcal D\mathcal J(K_P)) = E^*(BT^m)\big/I_{SR}^E \end{equation}

\

\nd where $I_{SR}^E$ is an ideal in $E^*(BT^m)$ described next. In
this context, the two-dimensional generators of the graded ring $E^*(BT^m)$ are
denoted by $\{v_1,v_2, \ldots, v_m\}$. The ideal  $I_{SR}^E$ is generated by 
all square-free monomials $v_{i_1}v_{i_2}\cdots v_{i_s}$ corresponding to 
$\{v_{i_1},v_{i_2}, \ldots, v_{i_s}\} \notin K_P.$
The ring (\ref{eqn:dj.definition}) is called the $E^*$-Stanley-Reisner ring.

For a toric manifold $M^{2n}$ an argument, based on the collapse of the Atiyah-Hirzebruch-
Serre spectral sequence for  (\ref{eqn:basic.fibration}), yields an isomorphism of $E^*$-algebras

\begin{equation}\label{eqn:cohomology}
E^*(M^{2n}) \cong E^*(\mathcal D\mathcal J(K_P))\big{/}J^E
\end{equation}

\

\nd where the ideal $J^E$ is generated by $p^*(E^2(BT^n))$ and therefore by the $E$-theory
Chern classes of certain associated line bundles, (\cite{buchstaber.ray}, page 18 and
\cite{civan.ray}, page 6).

For the case of non-singular compact projective toric varieties and $E$ equal to ordinary singular 
cohomology with integral coefficients ($E = H\mathbb{Z}$),
this is the celebrated result of Danilov and Jurkewicz, \cite{danilov}. The $E = H\mathbb{Z}$ version 
for the topological generalization of compact smooth toric varieties, is due to Davis and Januszkiewicz
\cite{davis.jan}. For certain singular toric varieties, the results of \cite{danilov} cover also the case  
$E = H\mathbb{Q}$.

The question of an analogue of (\ref{eqn:cohomology}) for a non-complex-oriented theory 
arises naturally. The obvious first candidate is $KO$-theory. 
The ring structure of 

\begin{equation}\label{eqn:product}
KO^*(BT^m) = 
KO^{*}(\prod_{i=1}^{m}{\mathbb{C}P^{\infty}})
\end{equation}

\

\nd does not seem to appear in the literature for $m > 2$. The thesis of Dobson
\cite{dobson} deals with the case $m = 2$. The $KO^{*}$-algebra structure of 
$KO^{*}(\mathbb{C}P^n)$ and $KO^{*}(\mathbb{C}P^\infty)$ appears explicitly for the 
first time  in \cite{civan.ray} where it is discussed in the context of a theorem of  Wood
which is mentioned in Section \ref{sect:basic.calculation} below.

Two different presentations for the ring $KO^*(BT^m)$ are given in 
Sections \ref{sec:ring.structure} and \ref{sec:ijays}.
Following Atiyah and Segal \cite{atiyah.segal}, these provide a
description of the completion of the representation ring $RO(T^m)$ at the
augmentation ideal. The calculation herein may be interpreted in
that context, along the lines of Anderson \cite{anderson}. In particular, the fact
that the complexification map and the realification map

\begin{align}\label{eqn:c}
c\colon KO^*(BT^m) \;&\longrightarrow\; KU^*(BT^m) \end{align}

\begin{align}\label{eqn:r}
r\colon KU^*(BT^m) \;&\longrightarrow\; KO^*(BT^m) 
\end{align}

\

\nd are injective and surjective respectively, is used throughout. 
This follows from the Bott exact sequence 

\
\begin{equation}\label{eqn:bott.seq}
\ldots \xrightarrow{} KO^{*+1}(X) \xrightarrow{\cdot e} KO^{*}(X) 
\xrightarrow{\chi} K^{*+2}(X) \xrightarrow{r} KO^{*+2}(X) \xrightarrow{} \ldots
\end{equation}
\

\nd where $\chi$ is complexification (\ref{eqn:c}) followed by multiplication by $v^{-1}$ the
Bott element. Since $KO^{*}(BT^m)$ is concentrated in even degree, the Bott sequence
implies that the realification map $r$ is surjective and complexification $c$ is injective.
They are related by
\begin{equation}\label{eqn:r.and.c}
(r\circ c)(x) = 2x \quad \text{and} \quad (c \circ r)(x) = x + \overline{x}.
\end{equation}

The complexity of the calculation is
a result of  the fact that the realification map $r$ is {\em not \/}a ring homomorphism. The
first presentation generalizes the methods of \cite{dobson}. A companion result is given
for $KO^{*}(\bigwedge_{i = 1}^{m}\mathbb{C}P^{\infty})$. 
The second approach produces generators better suited to the task of giving a description of 
$KO^*(\mathcal D\mathcal J(K_P))$ in terms of $KO^*(BT^m)$. The results of  \cite{bb}
are used then to give a description of $KO^*(M^{2n})$ analogous to (\ref{eqn:cohomology})
for any toric manifold which has no $\sq^2$-homology.

The group structure of $KO^{*}(\bigwedge_{i = 1}^{m}\mathbb{C}P^{\infty})$ is much more 
accessible than the ring structure. The Adams
spectral sequence yields a concise description in terms of the groups
$KU^{*}(\bigwedge_{i = 1}^{k}\mathbb{C}P^{\infty})$ with fewer smash product factors. This is 
discussed in the next section.










\
\section{The group structure of $KO^{*}(\bigwedge_{i = 1}^{m}\mathbb{C}P^{\infty})$}
\label{sect:basic.calculation}

\subsection{The $ko$-homology}\label{sec:ko.homology}
The calculation begins with the determination of $ko_{*}(BT^m)$,
the connective $ko$-homology corresponding to the spectrum $bo$.
The main tool is the Adams spectral sequence. It is used in conjunction with the following 
equivalence, which is well known among homotopy theorists and extends a result of Wood. 
Let  $bu$ denote the spectrum corresponding to connective complex $k$-theory. 

\begin{thm} \label{thm:bu.and.bo}
There is an equivalence of spectra

$$\bigvee_{k = 0}^{\infty} \Sigma^{4k+2}{bu}\; \longrightarrow\;
bo \wedge \mathbb{C}P^{\infty}.$$
\end{thm}

\

\begin{proof} Background information about the Adams spectral sequence in connection 
with $ko$-homology calculations may be found in \cite{bb} or \cite{baybrun}. A change of rings 
theorem implies that the  $E_2$-term of the Adams spectral sequence for  
$ko_{*}(\mathbb{C}P^{\infty}) = \pi_{*}(bo \wedge \mathbb{C}P^{\infty})$ depends on the 
$\mathcal{A}_1$-module structure of $H^{*}(\mathbb{C}P^{\infty}; \mathbb{Z}_{2})$
where $\mathcal{A}_1$ denotes the sub-algebra of the Steenrod algebra $\mathcal{A}$
generated by $\sq^1$ and $\sq^2$.

As an $\mathcal{A}_{1}$-module, $H^{*}(\mathbb{C}P^{\infty}; \mathbb{Z}_{2})$ is a sum
of shifted copies of $H^{*}(\mathbb{C}P^{2}; \mathbb{Z}_{2})$. Consequently, the $E_2$ term of the
spectral sequence is a sum of shifted copies of 
$\mbox{Ext}_{\mathcal{A}_{1}}^{s,t}(H^{*}(\mathbb{C}P^{2}; \mathbb{Z}_{2}), \mathbb{Z}_{2})$
and so has classes in even degree only. Hence, the spectral sequence collapses. A non-trivial
class in each $\pi_{4k+2}(bo \wedge \mathbb{C}P^{\infty})$ is represented in the
$E_2$-term by a generator of dimension $4k+2$ in filtration zero. The $\eta$-extension  on
this class is trivial as the $E_2$-term is zero in odd degree. So the map

\begin{equation}\label{eqn:bottom.cell}
S^{4k+2} \longrightarrow bo \wedge \mathbb{C}P^{\infty}\end{equation}

\

\nd extends over a $(4k+4)$-cell $\;e^{4k+4}$ attached by the Hopf map $\eta$. This gives a map

\begin{equation}\label{eqn:eta}
\Sigma^{4k}{\mathbb{C}P^{2}} = S^{4k+2} \cup_{\eta} e^{4k+4}
\longrightarrow bo \wedge \mathbb{C}P^{\infty}\end{equation}

\nd which extends to

\begin{equation}\label{eqn:splitting}
s\colon \bigvee_{k=0}^{\infty}\Sigma^{4k}{\mathbb{C}P^{2}}\; \longrightarrow \;
bo \wedge \mathbb{C}P^{\infty}.\end{equation}

\

\nd Smashing with $bo$ and composing with the product map 
$bo \wedge bo\; \xrightarrow{\mu}\; bo$ gives

\begin{equation}\label{eqn:bo.smash}
bo \wedge (\bigvee_{k=0}^{\infty}\Sigma^{4k}{\mathbb{C}P^{2}})\;
\xrightarrow{1\wedge s}\; bo \wedge bo \wedge \mathbb{C}P^{\infty} 
\; \xrightarrow{\mu \wedge 1} \;
bo \wedge \mathbb{C}P^{\infty}.\end{equation}

\

\nd The equivalence of spectra  $\Sigma^{2}{bu} \rightarrow bo \wedge {\mathbb{C}P^{2}}$, 
due to Wood and cited in \cite{jfa}, is used next to give a map

\begin{equation}\label{eqn:main.map}
g \colon\bigvee_{k = 0}^{\infty} \Sigma^{4k+2}{bu}
\longrightarrow
bo \wedge \mathbb{C}P^{\infty}.\end{equation}

\

\nd This map induces an isomorphism of stable homotopy groups and hence
gives the required equivalence of spectra. \end{proof}

\begin{rem}
An equivalence of the form \ref{eqn:bo.smash} follows also from the methods of  \cite{mahowald.ray}
and the fact that twice the Hopf bundle over ${\mathbb{C}}P^{\infty}$ is a Spin bundle and therefore
$bo$-orientable.
\end{rem}

The next corollary follows immediately.

\begin{cor}\label{cor:ko.homology}
There is an isomorphism of graded abelian groups

$$\bigoplus_{k=0}^{\infty} \Sigma^{4k+2}(\widetilde{ku}_{*}(\bigwedge_{i = 2}^{m}\mathbb{C}P^{\infty})) \; \longrightarrow \;
\widetilde{ko}_{*}(\bigwedge_{i = 1}^{m}\mathbb{C}P^{\infty}).$$

\end{cor}

\

\nd Notice here that the summands on the left hand side are  the underlying groups of 
a tensor product of divided power algebras each of which is dual to a polynomial algebra.

\

Recall next that there are classes $e \in ko_1, \; \alpha \in ko_4, \; \beta \in ko_8$ so that

\begin{equation}\label{eqn:ko.gens}
ko_{*} = \mathbb{Z}[e, \alpha, \beta]\big/\langle2e, e^{3}, e\alpha, \alpha^2 - 4\beta\rangle
\end{equation}

\

\nd and a class $v \in ku_{2}$ so that

\begin{equation}\label{eqn:ku.gens}
ku_{*} = \mathbb{Z}[v]\end{equation}

\

\begin{remark}\label{rem:module.structure}
An examination of the $E_2$-term of the Adams spectral sequence for  
$ko_{*}(\mathbb{C}P^{2})$ reveals that the action of $ko_{*}$ on $ko_{*}(\mathbb{C}P^{2}) \cong ku_{*}$ is given by

\begin{equation}\label{eqn:module.relations}
e\cdot 1 = e\cdot v = 0,\; \alpha\cdot 1= 2v^2, \; \beta \cdot 1 = v^4
\end{equation}

\

\nd This coincides with the module action of $ko_{*}$ on $ku_{*}$ given by the 
``complexification'' map, (\cite{dobson}, page 16). \end{remark}




\


\subsection{From $ko$-homology to $KO$-cohomology.}
One consequence of the calculation above is that the Bott element 
$\beta$ acts as a monomorphism on $ko_{*}(\bigwedge_{i = 1}^{s}\mathbb{C}P^{\infty})$ and so can 
be inverted to get the periodic $KO$-homology of $\bigwedge_{i = 1}^{s}\mathbb{C}P^{\infty}$.

\begin{prop}\label{prop:ko.to.KO} There is an isomorphism of abelian groups

$$\bigoplus_{k=0}^{\infty} \widetilde{KU}_{*\; + \;4k + 2}(\bigwedge_{i = 2}^{m}\mathbb{C}P^{\infty})\longrightarrow \widetilde{KO}_{*}(\bigwedge_{i = 1}^{m}\mathbb{C}P^{\infty}).$$

\end{prop} 

\begin{proof} The result follows from Corollary \ref{cor:ko.homology} and Remark
\ref{rem:module.structure} \end{proof}

\

\nd The proof of Theorem \ref{thm:bu.and.bo} works equally well in the dual situation. Let
$D(\mathbb{C}P^{2n})$ denote the $S$-dual of $\mathbb{C}P^{2n}$.
Aside from a dimensional shift,  $H^{*}(D(\mathbb{C}P^{2n}); \mathbb{Z}_{2})$,
as an $\mathcal{A}_{1}$-module,  is isomorphic to a sum of suspended copies of 
$H^{*}(\mathbb{C}P^{2}; \mathbb{Z}_{2})$. So, the Adams spectral sequence for
 $\pi_{*}(bo \wedge D(\mathbb{C}P^{2n}))$ collapses for dimensional reasons. The argument
 of Theorem \ref{thm:bu.and.bo} goes through essentially unchanged to give an equivalence of 
 spectra
 
\begin{equation}\label{eqn:reg.dual} 
g\colon \bigvee_{k = 1}^{n} \Sigma^{-4k}{bu}  \longrightarrow bo \wedge D(\mathbb{C}P^{2n}).
\end{equation}

\

\nd The next lemma, which follows directly from the discussion in Section 
\ref{sec:ko.homology}, records the fact that
(\ref{eqn:reg.dual}) is natural with respect to the inclusion 

$$CP^{2n} \stackrel{\subset}{\longrightarrow}  CP^{2(n+1)}.$$

\begin{lem}\label{lem:naturality}
The  following diagram commutes

\begin{equation}\label{cd:inclusions.diagram}
\begin{CD}
\bigvee\limits_{k = 1}^{n} \Sigma^{-4k}{bu} @>{g}>> 
bo \wedge D(\mathbb{C}P^{2n})\\
@AA{\phi}A                                          @A{\psi}AA \\
\bigvee\limits_{k = 1}^{n+1} \Sigma^{-4k}{bu} @>{g}>> bo \wedge D(\mathbb{C}P^{2(n+1)})\\\
\end{CD}
\end{equation}

\

\nd where the map $\phi$ collapses $\Sigma^{-4(n+1)}{bu}$ to a point and $\psi$ is
induced by the inclusion 
$$CP^{2n}  \stackrel{\subset}{\longrightarrow}  CP^{2(n+1)}.$$
\hfill  $\square$
\end{lem}

\

The duality result from \cite[Proposition 5.6]{jfa}, implies

\begin{equation}\label{eqn:smash.duality}
D(\bigwedge_{i = 1}^{m}\mathbb{C}P^{2n}) \; \simeq \; \bigwedge_{i = 1}^{m}D(\mathbb{C}P^{2n}).
\end{equation}

\nd From this follows an isomorphism of abelian groups, analogous to Proposition \ref{prop:ko.to.KO} 
for finite projective spaces,

\begin{equation}\label{eqn:homology.of.duals}
\bigoplus_{k=1}^{n} \widetilde{KU}_{* -4k}\big(D(\bigwedge_{i = 2}^{m}\mathbb{C}P^{2n})\big)\;
\longrightarrow \;
\widetilde{KO}_{*}\big(D(\bigwedge_{i = 1}^{m}\mathbb{C}P^{2n})\big)
\end{equation}

\

\nd and so an isomorphism 

\begin{equation}\label{eqn:smash.finite}
\bigoplus_{k=1}^{n} \widetilde{KU}^{*\;+\;4k}(\bigwedge_{i = 2}^{m}\mathbb{C}P^{2n})
\;\longrightarrow\;
\widetilde{KO}^{*}(\bigwedge_{i = 1}^{m}\mathbb{C}P^{2n}).
\end{equation}

\

The next result extends (\ref{eqn:smash.finite})  to  
$\bigwedge_{i = 1}^{m}\mathbb{C}P^{\infty}$.
 
\ 
 
\begin{prop}\label{prop:inverse.limit}
There are isomorphisms
$$\widetilde{KO}^{*}(\bigwedge_{i = 1}^{m}\mathbb{C}P^{\infty}) \quad \cong \quad
\lim_{\stackrel{\longleftarrow}{n}}{\widetilde{KO}^{*}(\bigwedge_{i = 1}^{m}\mathbb{C}P^{2n})}$$
\nd and 
$$\bigoplus_{k=1}^{\infty} \widetilde{KU}^{*+4k}(\bigwedge_{i = 2}^{m}\mathbb{C}P^{\infty}) 
\quad \cong \quad
\lim_{\stackrel{\longleftarrow}{n}}
\bigoplus_{k=1}^{n} \widetilde{KU}^{*+4k}(\bigwedge_{i = 2}^{m}\mathbb{C}P^{2n}).$$
\end{prop}

\begin{proof} It follows from the calculations above
that the maps in the inverse limit arising from

$$\widetilde{KO}^{*}\big(\bigwedge_{i = 1}^{m}\mathbb{C}P^{2(n+1)}\big)
\longrightarrow \widetilde{KO}^{*}\big(\bigwedge_{i = 1}^{m}\mathbb{C}P^{2n}\big)$$ 

\

\nd and

$$\bigoplus_{k=1}^{n+1} \widetilde{KU}^{*+4k}(\bigwedge_{i = 2}^{m}\mathbb{C}P^{2(n+1)})
\longrightarrow
\bigoplus_{k=1}^{n} \widetilde{KU}^{*+4k}(\bigwedge_{i = 2}^{m}\mathbb{C}P^{2n})$$

\

\nd (induced from the maps $\psi$ and $\phi$  of diagram (\ref{cd:inclusions.diagram}))
are all surjective. Thus the  Mittag-Leffler condition is satisfied and the 
$\ds{\lim_{\stackrel{\longleftarrow}{n}}{ }\!^{1}}$ terms are zero. \end{proof}

Finally, Lemma \ref{lem:naturality} implies that the isomorphisms (\ref{eqn:smash.finite})
are compatible with the maps in the inverse limits and so yield the main result of this section.

\begin{thm}\label{thm:main.theorem} There is an isomorphism of graded abelian groups

$$\bigoplus_{k=1}^{\infty} \widetilde{KU}^{*\;+\;4k}(\bigwedge_{i = 2}^{m}\mathbb{C}P^{\infty})
\longrightarrow
\widetilde{KO}^{*}(\bigwedge_{i = 1}^{m}\mathbb{C}P^{\infty}).$$
\end{thm}

\

\section{The algebra $KO^{*}(BT^{m})$}\label{sec:ring.structure}
This section contains the first of two descriptions of the the algebra $KO^{*}(BT^{m})$. It
extends the calculation done in \cite{dobson} for the case $m=2$. 

\subsection{Notation and statement of results} Here, as in Section
\ref{sec:introduction},

$$ BT^m \;\cong\; \prod_{i=1}^{m}{\mathbb{C}P^{\infty}}$$

\

The two sets of generators presented for $KO^{*}(BT^{m})$ have contrasting advantages.
The first description yields generators which are slightly complicated but
the relations among them are fairly straightforward. This situation is reversed in
the second description.

\

\nd The complexification and realification maps, (\ref{eqn:c}) and (\ref{eqn:r}),
are denoted again by $c$ and $r$ respectively.

\

Let $\alpha \in KO^{-4}$ and $\beta \in KO^{-8}$ be
the elements arising from (\ref{eqn:ko.gens}), for which  $\alpha^2 = 4\beta$. Let $v \in KU^{-2}$
be the Bott element, which satisfies $r(v^2) = \alpha$ and $c(\alpha) = 2v^2$.
The generators of $KU^{0}(BT^m)$ are denoted by $x_i$ for
$i = 1, \ldots , m$\; so that $KU^{0}(BT^m) \; \cong \; 
\mathbb{Z}[\mspace{-2mu}[ x_1,\ldots,x_m]\mspace{-2mu}]$.

\

\nd More notation is established next.

\begin{defin}\label{defin:don.generator.notation}

Consider  the set $N = \{1, \ldots,m\}$ and let $S \subseteq N$.
 Then

\begin{enumerate}
  \item set $\mathrm{min}(S) = \mathrm{min}\{i : i \in S\}$,
  \item let $|S|$ denote the cardinality of $S$,
  \item for $s \in \{0,1,2\}$, let $X_{S}^{(s)} = r(v^{s}
  \ds{\prod\limits_{i \in S}} x_i) \in KO^{-2s}(BT^m)$,
  \item let $X_{S} = X_{S}^{(0)}$ and $X_i = X_{\{i\}} = r(x_{i})$,
  \item for $s \in \{0,1\}$, let $X_{\phi}^{(s)} = 1 + (-1)^s$,
  \item for $s \in \{0,1\}$, let $M_{S}^{(s)} = X_{S}^{(s)}\cdot\!\!  
  \ds{\prod\limits_{i \in N \setminus S}}{X_{i}}$ and $M_S = M_{S}^{(0)}$.
  \item $\widehat{BT}^m =  \bigwedge_{i=1}^{m} \mathbb{C}P^{\infty}$ and
  \item  $\overline{{\mathbb{Z}}}[\mspace{-2mu} [-]\mspace{-2.1mu}]$ denotes the 
  augmentation  ideal of a   power series ring.
\end{enumerate}
\end{defin}
 
 \

\begin{thm} \label{thm:ko.prod.proj.basis.1}
There is an isomorphism of graded rings
$$KO^{*}(BT^m) \;\cong\; {\mathbb{Z}}[\gamma^{\pm 1}] \otimes
\overline{{Z}}
[\mspace{-2mu} [X_{S}, X_{S}^{(1)} : \phi \neq S \subseteq N ]\mspace{-2.1mu}]\big/\!\sim$$
where $\gamma$ is an element with $|\gamma| = -4$ satisfying 
$2\gamma = \alpha$ and $\gamma^{2} = \beta$. Here $\sim$ refers to the
two families of relations {\rm(I)} and {\rm (II)} below.
\begin{itemize}
\item[(I)] If $A$, $B$ and $C$ are disjoint subsets of $N$ and $0 \leq s, t \leq 1$, 
then
$$X_{A\cup B}^{(s)}X_{A\cup C}^{(t)} = \prod_{i \in A}X_{i}\cdot
\Biggl[\sum_{T \subseteq A}{X_{T\cup B \cup C}^{(s+t)}} +
 (-1)^{s + |A \cup B|} \sum_{S \subseteq B}{(-1)^{|S|}\biggl{(}\prod_{i \in S} X_{i}\biggr{)}
 X_{C \cup B \setminus S}^{(s+t)}}\Biggr].$$
\nd Here $B$, $C$, $S$,  $T$ may be empty, $X_{S}^{(2)} = \gamma X_{S}$
and products over empty sets are considered equal to $1$.
\item[]
\item[(II)] For $ i < {\mathrm {min}}(S)$, $|S| >1$ and $s \in \{0,1\}$,
$$X_{i}X_{S}^{(s)} = (-1)^{s}\sum_{T \subseteq S}
{\Biggl[(-1)^{|T|} \biggl(\prod_{j \in S \setminus T}X_{j}\biggr) \cdot X_{\{i\} \cup T}^{(s)}\Biggr]} + X_{\{i\} \cup S}^{(s)}\;.$$
\nd Again, $T$ may be empty.
\end{itemize}
\end{thm}
 
 \
  
\begin{remark}
The element $\gamma$ is introduced here for notational convenience. It arises naturally
in the Adams spectral sequence and has the property that $\gamma{r(x)} = r(v^2x)$ for
all $x \in KU^{0}(BT^m)$. The use of $\gamma$ may be removed in Theorem
\ref{thm:ko.prod.proj.basis.1} and in Corollary \ref{cor:ko.prod.proj.basis.1} (below), by 
allowing the choice of the exponent $s$ in Definition \ref{defin:don.generator.notation} to be unrestricted.
\end{remark}  
  
\

Relations (I) allow the elimination of all products $X_{U}^{(u)}X_{V}^{(v)}$
with $|U|$ and $|V|$ both greater than $1$, reducing everything to products of
$X_{i}$'s times at most one $X_{W}^{(w)}$ with $|W| > 1$. Relations (II) allow the 
elimination of $X_{i}X_{W}^{(w)}$ with $|W| > 1$ and $i < {\mathrm {min}}(W)$. 
Notice that for a product 
$X_{i_{1}}X_{i_{2}}\cdot \ldots \cdot X_{i_{k}}X_{W}^{(w)}$, (II)
need  be performed once only for just one $X_{i_{j}}$ with minimal $i_{j}$.

\

The next corollary is now immediate.

\begin{cor} \label{cor:ko.prod.proj.basis.1}
Every element of $KO^*(BT^m)$ can be expressed as a formal sum of terms
from
$$\mathcal{G}_1 = \bigg\{\gamma^{j}\Bigl(\prod_{i \;= \;{\mathrm{min}}(S)}^{m} X_{i}^{e_{i}}
\Bigr)X_{S}^{(s)}: S \subseteq N, \;S \neq \phi, \;e_i \geq 0, \;j \in \mathbb{Z}\; 
\text{and} \; s \in \{0,1\}\bigg\}.$$
\end{cor}

\

The example below follows easily from Theorem \ref{thm:ko.prod.proj.basis.1}.

\begin{exm} \label{exm:corrected.don.dobson.basis}
For $s \in \{0,1\}$, $KO^{-(4j+2s)}(BT^2)$ has a basis
$$\mathcal{G}_1 = \left\{\gamma^{j}X_{2}^{e_2}X_{2}^{(s)},\;
\gamma^{j}X_{1}^{e_1}X_{2}^{e_2}X_{1}^{(s)},\;
\gamma^{j}X_{1}^{e_1}X_{2}^{e_2}X_{\{1,2\}}^{(s)} : e_1, e_2 \geq 0\right\}.$$

\nd The following relations determine all products among these basis elements.
Here $s \in\{0,1\}$ and $i \in \{1,2\}$. Recall $X_{S}^{(2)} = \gamma X_{S}$.

$$\ds{\begin{array}{lll}
X_{\{1,2\}}X_{\{1,2\}} & = & X_1X_2(X_{\{1,2\}} + X_1 + X_2  + 4)\\
\\
X_{\{1,2\}}^{(s)}X_{\{1,2\}}^{(1)} & = & X_1X_2(X_{\{1,2\}}^{(s+1)} + X_{1}^{(s+1)}
+ X_{2}^{(s+1)})\\
\\
X_{i}^{(1)}X_{i}^{(1)} & = & \gamma(X_{i}^2 + 4X_i)\\
\\
X_{1}^{(s)}X_{2}^{(1)} & = & 2X_{\{1,2\}}^{(s+1)} - X_{2}X_{1}^{(s+1)}\\
\\
X_{1}^{(1)}X_{\{1,2\}}^{(1)} & = & \gamma X_1(2X_2 + X_{\{1,2\}}).
\end{array}}$$
\end{exm}

\

\nd The first two relations are of type (I); the last three are of type (II).
\begin{rem}\label{rem:corrected.dobson.remark} The case $m = 2$ is done in
\cite{dobson}. The example above agrees with Proposition 8.2.20 in \cite{dobson} after
certain typographical errors are corrected. (These include replacing all the equal
signs with minus signs and correcting the formula for ``$w_{2i}w_{2j}$'' so that 
it is consistent with Lemma 8.2.8 in the same document.)
\end{rem}

\

A closely-related result gives $KO^{*}(\widehat{BT}^m)$.

\

\begin{thm} \label{thm:ko.of.smash}
$\widetilde{KO}^{-(4j+2s)}(\widehat{BT}^m)$
is a free module over
${\mathbb Z}[\mspace{-2mu} [X_1, \ldots , X_{m}]\!]$ on
$$\left\{ \gamma^{j}M_{S}^{(s)} : 1 \in S \subseteq N \right\} $$
The product $M_{S_{1}}^{(s)}M_{S_{2}}^{(t)}$ can be computed in
terms of this basis from the relations $\mathrm{(I)}$ and $\mathrm{(II)}$
of Theorem \ref{thm:ko.prod.proj.basis.1}.
\end{thm}

\

Relations (I) and (II) of Theorem \ref{thm:ko.prod.proj.basis.1} are proved next.
This is followed by an identification of the terms given by Theorem 
\ref{thm:ko.of.smash} with those appearing in Theorem \ref{thm:main.theorem}.
Finally, Theorem \ref{thm:ko.prod.proj.basis.1} is derived from 
Theorem \ref{thm:ko.of.smash}.

\

\subsection{The proof of relations (I) and (II)}
The complexification map $c$ is injective and so it suffices to prove
relations (I) and (II) after $c$ is applied. For convenience, the relations will be 
verified in the ring $KU^{*}(BT^m)$ with the classes 
$\{z_i = \sqrt{1 + x_i} : i = 1, \ldots , m\}$ adjoined.

\

\begin{align*} 
c(X_{S}^{(s)}) & = v^s\prod_{i \in S}{x_i} + \overline{v^s}
\prod_{i \in S}{\overline{x}_{i}}\\ 
\\
& = v^s(\prod_{i \in S}{x_i})\Bigl(1 + (-1)^{s + |S|}
\prod_{i \in S}{\frac{1}{1 + x_i}\Bigr)} \\
\\
& =  v^s\prod_{i \in S}{\Bigl(\frac{x_i}{\sqrt{1 + x_i}}\Bigr)}\cdot
\Biggl(\prod_{i \in S}{\sqrt{1 + x_i}} + (-1)^{s + |S|}\prod_{i \in S}
{\frac{1}{\sqrt{1 + x_i}}}\Biggr)\\
\\
& =  v^s\prod_{i \in S}{\Bigl(\frac{z_{i}^2 -1}{z_i}\Bigr)}\cdot
\Biggl(\prod_{i \in S}{z_i} + (-1)^{s + |S|}\prod_{i \in S}
{\frac{1}{z_i}}\Biggr)\\
\\
& =  v^s\prod_{i \in S}{\Bigl(z_i - \frac{1}{z_i}\Bigr)}\cdot
\Biggl(\prod_{i \in S}{z_i} + (-1)^{s + |S|}\prod_{i \in S}
{\frac{1}{z_i}}\Biggr).\\
\end{align*} 







\

More notation is introduced next.
\begin{defin}\label{defin:don.relations.notation}

Let $A$, $S$, and $T$ be disjoint subsets of $N$.
 Then

\begin{enumerate}
  \item Let 
  $$w_{A,S,T}^{(s)} \;\; := \; \; \frac{(\prod\limits_{j \in A}z_{j}^{2})(\prod\limits_{j \in S}z_j)}
   {\prod\limits_{j \in T}z_j} + (-1)^{s +|S \cup T|}
   \frac{\prod\limits_{j \in T}z_{j}}{(\prod\limits_{j \in A}{z_j}^{2})(\prod\limits_{j \in S} z_j)}$$
   \item[]
   \item for $A = \phi$, set $w_{S,T}^{(s)} = w_{\phi,S,T}^{(s)}$ and 
   for $T = \phi$, $w_{S}^{(s)} = w_{S,\phi}^{(s)}$ 
   \item[]
   \item set $w_{i} = w_{\{i\}}^{(0)} = z_{i} - \ds{\frac{1}{z_i}}$.
 \end{enumerate}
\end{defin}

\nd Notice that in this new notation, the calculation above is 

$$c(X_{S}^{(s)}) = v^s\bigl(\prod_{i \in S}{w_i}\bigr)w_{S}^{(s)}.$$

\

Recall that If $A$, $B$ and $C$ are disjoint subsets of $N$ and $0 \leq s, t \leq 1$, 
relation (I) is

\begin{equation}\label{equation:relation.(I)}
X_{A\cup B}^{(s)}X_{A\cup C}^{(t)} = \prod_{i \in A}X_{i}\cdot
\Biggl[\sum_{T \subseteq A}{X_{T\cup B \cup C}^{(s+t)}} +
 (-1)^{s + |A \cup B|} \sum_{S \subseteq B}{(-1)^{|S|}\biggl{(}\prod_{i \in S} X_{i}\biggr{)}
 X_{C \cup B \setminus S}^{(s+t)}}\Biggr].
\end{equation}

\

\

\nd Applying $c$ and dividing both sides by 
$v^{s+t}\bigl(\prod\limits_{i \in A}{w_{i}^2}\bigr)
\bigl(\prod\limits_{i \in B\cup C}{w_{i}}\bigr)$ 
makes \eqref{equation:relation.(I)} equivalent to

\

\begin{equation}\label{equation:relation.(I).in.terms.of.w}
w_{A\cup B}^{(s)}w_{A\cup C}^{(t)} = 
\sum_{T \subseteq A}{\Bigl[\bigl(\prod_{i\in T}{w_i}\bigr)w_{T \cup B \cup C}^{(s+t)}\Bigr]}
+ (-1)^{s+ |A \cup B|}\sum_{S \subseteq B}{(-1)^{|S|}\bigr(\prod_{i \in S}{w_i}\bigr)
w_{C \cup B \setminus S}^{(s+t)}}.
\end{equation}

\

\nd The left hand side of  \eqref{equation:relation.(I).in.terms.of.w} is checked
easily to satisfy

\

\begin{equation}\label{equation:lhs.of.relation(I)}
w_{A\cup B}^{(s)}w_{A\cup C}^{(t)} =  w_{A, B\cup C, \phi}^{(s+t)}
+ (-1)^{s + |A \cup B|}w_{C,B}^{(s+t)}.
\end{equation}

\

\nd The first term on the right hand side of \eqref{equation:relation.(I).in.terms.of.w} satisfies

\

\begin{equation}\label{equation:rhs.first.term}
\sum_{T \subseteq A}\Bigl[{\bigl(\prod_{i\in T}{w_i}\bigr)w_{T \cup B \cup C}^{(s+t)}
\Bigr]} =
\sum_{T \subseteq A}{\Bigl(\sum_{R \subseteq T}{(-1)^{|T\setminus R|}
w_{R,B\cup C,\phi}^{(s+t)}}\Bigr),}
\end{equation}

\

\nd where $R \subseteq T$ is defined by the fact that $T\setminus R$ is
the set of $i$'s in $T$ for which,  in $\prod\limits_{i\in T}{w_i}$, is chosen
the second term of $w_i = z_i - \frac{1}{z_i}$. With  
$T$ satisfying $R \subseteq T \subseteq A$, the order of summation on the 
right hand side of \eqref{equation:rhs.first.term} is rearranged 
to give

\

\begin{equation}\label{equation:rhs.of.relation(I).first.term}
\sum_{T \subseteq A}{\Bigl(\sum_{R \subseteq T}{(-1)^{|T\setminus R|}
w_{R,B\cup C,\phi}^{(s+t)}}\Bigr)} = 
\sum_{R\subseteq A}{w_{R, B\cup C,\phi}^{(s+t)}\Bigl(\sum_
{T\subseteq A}{(-1)^{|T\setminus R|}}\Bigr)}.
\end{equation}

\

Now
\begin{equation}\label{equation:binomial.identity}
\sum_{T\subseteq A}{(-1)^{|T\setminus R|}} =
\sum_{j = 0}^{|A \setminus R|}{(-1)^{j}\biggl(\begin{array}{c}
|A \setminus R|\\
j
\end{array}\biggr)}.
\end{equation}

\

\nd Here, the binomial coefficient on the right hand side counts the number of
sets $T$, $R \subseteq T \subseteq A$ satisfying  $|T\setminus R| = j$. Notice that
the right-hand side is zero unless $A = R\;$ in which case it equals 1. So now \eqref{equation:rhs.first.term}
implies that the first term on the right of \eqref{equation:relation.(I).in.terms.of.w}
is $w_{A, B\cup C, \phi}^{(s+t)}$  which is the first term on the right-hand side of 
\eqref{equation:lhs.of.relation(I)}. 

\

The second term on the right hand side of \eqref{equation:relation.(I).in.terms.of.w} 
is analyzed similarly. With $S$ satisfying $U \subseteq S \subseteq  B$,

\begin{multline}\label{equation:rhs.of.relation(I).second.term}
(-1)^{s+ |A \cup B|}\sum_{S \subseteq B}{(-1)^{|S|}\bigr(\prod_{i \in S}{w_i}\bigr)
w_{C \cup B \setminus S}^{(s+t)}}\\  
= \; \; (-1)^{s+ |A \cup B|}\sum_{S \subseteq B}{\biggl((-1)^{|S|}
\sum_{U \subseteq S}{(-1)^{|U|}w_{C\cup B \setminus U,U}^{(s+t)}}\biggr)}
\end{multline}

\

\nd where here $S \setminus U$  is the set of $i$'s in $S$ 
for which is chosen in $\prod\limits_{i\in S}{w_i}$
the second term of $w_i = z_i - \frac{1}{z_i}$. Continuing as above,

\begin{equation}\label{equation:binomial.identity.second term}
(-1)^{s+ |A \cup B|}\sum_{S \subseteq B}{\biggl((-1)^{|S|}
\sum_{U \subseteq S}{(-1)^{|U|}w_{C\cup B \setminus U,U}^{(s+t)}}\biggr)}
\qquad \qquad \qquad \qquad
\end{equation}
\begin{align*}
\qquad \qquad & =(-1)^{s+ |A \cup B|}\sum_{U \subseteq B}
{\biggl[w_{C \cup B\setminus U,U}^
{(s+t)}\Bigl(\sum_{S}{(-1)^{|S \setminus U|}}\Bigr)\biggr]}\\
 \qquad \qquad & =(-1)^{s+ |A \cup B|}\sum_{U \subseteq B}
 \Biggl[w_{C \cup B\setminus U,U}^
{(s+t)}\sum_{j=0}^{|B \setminus U|}{(-1)^{j}}
           \biggl(\begin{array}{c}
                   |B \setminus U|\\
                   j
                   \end{array}\biggr)\Biggr]\\
                   \\
\qquad \qquad & =(-1)^{s+ |A \cup B|}w_{C,B}^{(s+t)}                   
\end{align*}

\

\nd which is the second term on the right hand side of 
\eqref{equation:lhs.of.relation(I)}. This completes the proof of the 
relations (I).

\ 

The verification of relations (II) is next.
For $ i < {\mathrm {min}}(S)$, $|S| >1$ and $s \in \{0,1\}$, the second set
of relations is

\

$$X_{i}X_{S}^{(s)} = (-1)^{s}\sum_{T \subseteq S}
{\Bigl[(-1)^{|T|} \bigl(\prod_{j \in S \setminus T}X_{j}\bigr) \cdot X_{\{i\} \cup T}^{(s)}\Bigr]} + X_{\{i\} \cup S}^{(s)}\;.$$

\

 \nd Applying $c$ to both sides  and dividing by $\prod\limits_{j \in \{i\} \cup S}{\!w_j}$
makes relations (II) equivalent to

\

\begin{equation}\label{equation:relation(II).in.terms.of.w}
w_{i}w_{S}^{(s)} = (-1)^{s}\sum_{T\subseteq S}\Bigl[(-1)^{|T|}w_{\{i\}\cup T}^{(s)}
\prod_{j \in S\setminus T}w_j\Bigr] + w_{\{i\}\cup S}^{(s)}.
\end{equation}

\

\nd Definition \ref{defin:don.relations.notation} implies immediately that


\

\begin{equation}\label{equation:lhs.of.relation(II)}
w_{i}w_{S}^{(s)} = - \;w_{S,\{i\}}^{(s)} \; +\; w_{\{i\} \cup S}^{(s)}.
\end{equation}

\ 

\nd and so it remains to show that

\

\begin{equation}
(-1)^{s}\sum_{T\subseteq S}\Bigl[(-1)^{|T|}w_{\{i\}\cup T}^{(s)}
\prod_{j \in S\setminus T}w_j\Bigr] = - \;w_{S,\{i\}}^{(s)}.
\end{equation}

\

 \nd To this end and using the fact that $i < {\mathrm{min}}(S)$ implies
$i \notin S$,

\

\begin{align*}\label{equation:first.times.of.relation(II)}
(-1)^{s}\sum_{T\subseteq S}\Bigl[(-1)^{|T|}w_{\{i\}\cup T}^{(s)}
\prod_{j \in S\setminus T}w_j\Bigr] & =
(-1)^{s}\sum_{T \subseteq S}{\Bigl[(-1)^{|T|}}\sum_{B \subseteq S\setminus T}
{(-1)^{|B|}w_{\{i\} \cup S\setminus B, B}^{(s)}}\Bigr]\\
\\
& = (-1)^{s}\sum_{B \subseteq S}{\Bigl[(-1)^{|B|}w_{\{i\}\cup S\setminus B,B}^{(s)}
\sum_{T \subseteq S\setminus B}{(-1)^{|T|}}\Bigr]}\\
\\
& = (-1)^{s}\sum_{B \subseteq S}{\Biggl[(-1)^{|B|}w_{\{i\}\cup S\setminus B,B}^{(s)}
\sum_{j = 0}^{S\setminus B}{(-1)^{j}}
 \biggl(\begin{array}{c}
                   |S \setminus B|\\
                   j
                   \end{array}\biggr)\Biggr]}\\
\\
& = (-1)^{s + |S|}w_{\{i\}, S}^{(s)} \; = \; - w_{S, \{i\}}^{(s)}
\end{align*}

\

\

\nd where, as in  \eqref{equation:binomial.identity},
  $\;\; \ds{\sum_{j = 0}^{S\setminus B}{(-1)^{j}}
 \biggl(\begin{array}{c}
                   |S \setminus B|\\
                   j
                   \end{array}\biggr) = 0}$ unless $S = B$ in which case it equals 1.
                   
\

\                   

\subsection{The proof of Theorems \ref{thm:ko.of.smash} and 
\ref{thm:ko.prod.proj.basis.1}.}  First, the additive generators appearing in
Theorem \ref{thm:ko.of.smash} are identified with the $KU$ generators given by 
Theorem \ref{thm:main.theorem}. Choose generators 
$y_i \in \widetilde{KU}^{0}(\bigwedge_{i = 2}^{m}\mathbb{C}P^{\infty})$ so that

\begin{equation}
\widetilde{KU}^{*}(\bigwedge_{i = 2}^{m}\mathbb{C}P^{\infty}) \;\cong\;
\mathbb{Z}[v^{\pm 1}][\mspace{-2mu}[ y_2,\ldots,y_m]\mspace{-2mu}]\cdot
(y_2\cdots y_m). \end{equation}

\

\nd Theorem \ref{thm:main.theorem}  can be written as

\begin{equation}\label{eqn:additive.structure}
\widetilde{KO}^{*}(\bigwedge_{i = 1}^{m}\mathbb{C}P^{\infty}) \;\cong\;
\bigoplus_{k=1}^{\infty} \widetilde{KU}^{*\;+\;4k}(\bigwedge_{i = 2}^{m}\mathbb{C}P^{\infty})
\;\cong\; \mathbb{Z}[v^{\pm 1}][\mspace{-2mu}[z, y_2,\ldots,y_m]\mspace{-2mu}]\cdot
(zy_2\cdots y_m)\end{equation}

\

\nd where $z$ is given grading equal to $4$.  The fact that the realification
map $r$ increases filtration in the Adams spectral sequence by $1$, $v$ has filtration 1
and $\gamma$ filtration $2$, allows the determination of  the Adams filtration of the
generators described in Theorem \ref{thm:ko.of.smash}. This then makes possible a 
comparison of generators via the Adams spectral sequence, modulo terms of higher filtration. 

\

\begin{lem}\label{lem:comparison}
Let $S$ be a set satisfying $1\in S \subseteq N$. In the description 
(\ref{eqn:additive.structure}), the element

$$v^{2j+s}z^{e_1}y_2^{2e_2}\cdot\ldots\cdot y_m^{2e_m}
\big(\prod_{i\in N\setminus S}{y_i}\big)zy_2\cdot\ldots\cdot y_m 
 \;\in\;
 \widetilde{KU}^{-(4j+2s) +\;4(e_1 +1)}(\bigwedge_{i = 2}^{m}\mathbb{C}P^{\infty})$$ 

\

\nd corresponds to the element 

$$X_{1}^{e_1}\cdot \ldots \cdot X_{m}^{e_m}\gamma^{j}M_{S}^{(s)} \; \in\;
\widetilde{KO}^{-(4j+2s)}(\widehat{BT}^m) \;=\; 
\widetilde{KO}^{-(4j+2s)}(\bigwedge_{i = 1}^{m}\mathbb{C}P^{\infty}),$$
 
\

\nd modulo terms of higher filtration in the Adams spectral sequence. \hfill $\square$
\end{lem}

\

\nd Lemma \ref{lem:comparison}  allows now the use of Theorem \ref{thm:main.theorem} 
to conclude that the generators given by Theorem \ref{thm:ko.of.smash} must span 
$\widetilde{KO}^{-(4j+2s)}(\widehat{BT}^m)$ and be linearly independent.

\

The proof of Theorem \ref{thm:ko.prod.proj.basis.1} from Theorem \ref{thm:ko.of.smash} 
uses the homotopy equivalence 

\begin{equation}\label{eqn:split.product}
\Sigma(Y_1 \times Y_2 \times \ldots \times Y_m) \longrightarrow 
\Sigma\Big(\bigvee_{S \subseteq N} \big(\bigwedge_{i\in S}{Y_i}\big)\Big)\end{equation}

\

\nd where $Y_i = \mathbb{C}P^{\infty}$ for all $i \in N$. Theorem \ref{thm:ko.of.smash} 
is applied to each wedge summand $\bigwedge\limits_{i\in S}{Y_i}$ to conclude that
for $S = \{i_1,i_2,\ldots,i_{|S|}\}$,
$\;\widetilde{KO}^{-(4j+2s)}\displaystyle{\big(\bigwedge\limits_{i\in S}{Y_i}\big)}$ is a free module over
${\mathbb Z}[\mspace{-2mu} [X_{i_1}, \ldots , X_{i_{|S|}}]\mspace{-2mu}]$ on
$\left\{ \gamma^{j}M_{T}^{(s)} : 1 \in T \subseteq S \right\} $. Assembling these generators
over all  $S\subseteq N$, $S \neq \phi$ produces the generators in 
Theorem \ref{thm:ko.prod.proj.basis.1}. The multiplicative relations (I) and (II) have been
checked.

\

\section{A second set of generators for the algebra $KO^{*}(BT^{m})$}\label{sec:ijays}
\subsection{Notation and statement of results} As usual,
\begin{align*}
KU^{*} \;& \cong \;  \mathbb{Z}[v, v^{-1}] \;\;  \mbox{with}\;\;  v \in KU^{-2}\\
KO^{*} \;& \cong \;  \mathbb{Z}[e , \alpha, \beta, \beta^{-1}]/(2e, e^3, e\alpha, 
\alpha^{2} - 4\beta) 
\end{align*} 

\nd where $e \in KO^{-1}, \alpha \in KO^{-4}$ and $\beta \in KO^{-8}$.
As in Section \ref{sec:ring.structure}, denote by  $x_i, \; i = 1, \ldots , n$,  the generators 
of $KU^{0}(BT^n)$. It is convenient to write

\begin{equation}\label{eqn:conjugates}
KU^{0}(BT^m) \; \cong \; 
\mathbb{Z}[\mspace{-2mu}[ x_1,\ldots,x_m,
\overline{x}_{1},\ldots,\overline{x}_{m}]\mspace{-2mu}]\big/(x_{i}\overline{x}_{i} + x_i + 
\overline{x}_{i}).\end{equation}

\

\nd Let $I = (i_1,i_2,\ldots,i_m) \; \mbox{and} \; J = (j_1,j_2,\ldots,j_m)$ \; with\;
$i_k \geq 0 , j_k \geq 0 \;\; \mbox{for} \;\; k = 1, \ldots ,m$.
For $s \in \mathbb{Z}$ and $r$ the realification map (\ref{eqn:r}), set

$$[I,J]^{(s)} \; := \; 
r\big(v^{s}x_{1}^{i_{1}}x_{2}^{i_{2}}\ldots x_{m}^{i_{m}}
(\overline{x}_{{1}})^{j_{1}}(\overline{x}_{{2}})^{j_{2}}\ldots (\overline{x}_{{m}})^{j_{m}}\bigr)$$

\

\nd  in $KO^{-2s}(BT^{m})$.  If $s = 0$, the notation $[I,J]$ is used instead of
$[I,J]^{(0)}$.

\

\begin{thm}\label{thm:ij.relations}
The classes $[I,J]^{(s)}$ satisfy the relations:

\skp{0.2}

\begin{itemize}
\item[(A)] $\quad \ds{[I,J]^{(s)} \; = \; (-1)^{s}[J,I]^{(s)}}$
\item[]
\item[(B)] $\quad \ds{[I,J]^{(s)} \; = \; -[I',J]^{(s)}\; - \; [I,J']^{(s)}}$
\skp{0.2}
\nd where, for $I = (i_1,\ldots,i_k, \ldots, i_m), \; J = (j_1,\ldots,j_k, \ldots, j_m)$
with $\;i_{k}\cdot j_{k} \neq 0$, 
$$I' =  (i_1,\ldots,i_{k}-1, \ldots, i_m)\quad  \text{and} \quad J' = (j_1,\ldots,j_{k}-1, \ldots, j_m).$$
\item[]
\item[(C)] $\quad [I,J]^{(s)}\cdot[H,K]^{(t)} = [I+H,J+K]^{(s+t)} +
(-1)^{s}[J+H,I+K]^{(s+t)}$
\skp{0.2}
\nd where the product here is in $KO^*(BT^m)$.
 
\end{itemize} 
\end{thm}

\begin{rem}
Formula (C) is symmetric because relation (A) implies
$$\ (-1)^{s}[J+H,I+K]^{(s+t)} = (-1)^{t}[I+K,J+H]^{(s+t)}.$$
\end{rem}

\nd {\em Proof of Theorem \ref{thm:ij.relations}.\/} Relations (A) follow immediately 
from (\ref{eqn:r.and.c}) by 
applying complexification then realification. Relations (B) follow by recalling 
that $x = -\overline{x}(1 + x)$ and decomposing
$x^{i}\overline{x}^{j-1}$ as 

\begin{align*}
x^{i}\overline{x}^{j-1}   &= x^{i-1}x\overline{x}^{j-1}\\
                                    &= x^{i-1}\overline{x}^{j-1}( -\overline{x}(1 + x))\\
                                    &= -x^{i-1}\overline{x}^{j} - x^{i}\overline{x}^{j}\\
\end{align*}

\nd which gives
$x^{i}\overline{x}^{j} = -x^{i-1}\overline{x}^{j} - x^{i}\overline{x}^{j-1}$.
To see relations (C), the complexification monomorphism $c$ is applied to both sides.

$$c\big([I,J]^{(s)}\cdot[H,K]^{(t)}\big) = 
c\big([I,J]^{(s)}\big)\cdot c\big([H,K]^{(t)}\big)$$
\begin{align*}
&\quad= [v^{s}x_{1}^{i_{1}}x_{2}^{i_{2}}\ldots x_{m}^{i_{m}}
(\overline{x}_{{1}})^{j_{1}}(\overline{x}_{{2}})^{j_{2}}\ldots (\overline{x}_{{m}})^{j_{m}} +
(-1)^{s}v^{s}(\overline{x}_{{1}})^{i_{1}}(\overline{x}_{{2}})^{i_{2}}\ldots (\overline{x}_{{m}})^{i_{m}}
x_{1}^{j_{1}}x_{2}^{j_{2}}\ldots x_{m}^{j_{m}}]\\
&\quad \cdot [v^{t}x_{1}^{h_{1}}x_{2}^{h_{2}}\ldots x_{m}^{h_{m}}
(\overline{x}_{{1}})^{k_{1}}(\overline{x}_{{2}})^{k_{2}}\ldots (\overline{x}_{{m}})^{k_{m}} +
(-1)^{t}v^{t}(\overline{x}_{{1}})^{h_{1}}(\overline{x}_{{2}})^{h_{2}}\ldots (\overline{x}_{{m}})^{h_{m}}
x_{1}^{k_{1}}x_{2}^{k_{2}}\ldots x_{m}^{k_{m}}]\\
\\
&\hspace{0.4in} = v^{s+t}x_{1}^{i_{1}+h_{1}}x_{2}^{i_{2}+h_{2}}\ldots x_{m}^{i_{m}+h_{m}}
(\overline{x}_{{1}})^{j_{1}+k_{1}}(\overline{x}_{{2}})^{j_{2}+k_{2}}\ldots 
(\overline{x}_{{m}})^{j_{m}+k_{m}}\\
&\hspace{0.6in} + (-1)^{s+t}v^{s+t}x_{1}^{j_{1}+k_{1}}x_{2}^{j_{2}+k_{2}}\ldots 
x_{m}^{j_{m}+k_{m}}
(\overline{x}_{{1}})^{i_{1}+h_{1}}(\overline{x}_{{2}})^{i_{2}+h_{2}}\ldots 
(\overline{x}_{{m}})^{i_{m}+h_{m}}\\
&\hspace{0.6in} + (-1)^{s}v^{s+t}x_{1}^{j_{1}+h_{1}}x_{2}^{j_{2}+h_{2}}\ldots 
x_{m}^{j_{m}+h_{m}}
(\overline{x}_{{1}})^{i_{1}+k_{1}}(\overline{x}_{{2}})^{i_{2}+k_{2}}\ldots 
(\overline{x}_{{m}})^{i_{m}+k_{m}}\\
&\hspace{0.6in} + (-1)^{t}v^{s+t}x_{1}^{i_{1}+k_{1}}x_{2}^{i_{2}+k_{2}}\ldots x_{m}^{i_{m}+k_{m}}
(\overline{x}_{{1}})^{j_{1}+h_{1}}(\overline{x}_{{2}})^{j_{2}+h_{2}}\ldots 
(\overline{x}_{{m}})^{j_{m}+h_{m}}\\
\\
&=\; c \big([I+H,J+K]^{(s+t)}\big) + c\big((-1)^{s}[J+H,I+K]^{(s+t)}\big) \\
&=\; c \big([I+H,J+K]^{(s+t)}\big) + c\big((-1)^{t}[I+K,J+H]^{(s+t)}\big). \qquad \qquad \qquad  
\qquad \qquad \qquad \qquad  \square 
\end{align*} 

\

\begin{remark}
The elements $X_{S}^{(s)}$ of Definition \ref{defin:don.generator.notation} are related to the
classes $[I,J]^{(s)}$ by
$$X_{S}^{(s)} = [(\epsilon(1), \epsilon(2), \ldots,\epsilon(m)), (0,0,\ldots,0)]^{(s)}$$

\nd where $\epsilon$ is the characteristic function of $S$.
\end{remark}

Next, a distinguished class of elements $[I,J]^{(s)} \in KO^{-2s}(BT^m)$ is selected. 

\nd For $I = (i_1,i_2,\ldots,i_m)$  and $J = (j_1,j_2,\ldots,j_m)$, with all $i_k \geq 0$, 
$j_k \geq 0$, set

\begin{equation}\label{eqn:g2}
\mathcal{G}_2 := \big\{[I,J]^{(s)} : I\cdot J = 0 \;\; \text{and}\; \;
i_l \geq j_l \; \;\text{if}\; \;i_k + j_k = 0 \;\;\text{for}\;\; k < l\big\}\end{equation}

\nd where here, $I\cdot J$ denotes the dot product of vectors. 

\

\nd The $KO^*$-module structure is  described easily. Recall that

\begin{align*}
KO^{*} \;& \cong \;  \mathbb{Z}[e , \alpha, \beta, \beta^{-1}]/(2e, e^3, e\alpha, 
\alpha^{2} - 4\beta).
\end{align*}

\

\begin{lem}\label{eqn:module.structure}
The $KO^*$-module action on $KO^*(BT^m)$ is given by
\begin{align*}
e\cdot ([I],[J])^{(s)} &= 0\\
\alpha\cdot ([I],[J])^{(s)} &= 2([I],[J])^{(s+2)}\\
\beta\cdot ([I],[J])^{(s)} &= ([I],[J])^{(s+4)}.
\end{align*}
\end{lem}

\begin{proof} The complexification monomorphism $c$ is applied to both sides of 
these relations. The result follows then from the identities $c(e) = 0, \;c(\beta) = v^4$
and $c(\alpha) = 2v^2$ from \cite{dobson}, Lemma 2.0.3.
\end{proof}

\

\begin{thm}\label{thm:second.set}
Every element of $KO^*(BT^m)$ can be expressed as a formal sum of terms from $\mathcal{G}_2$.
\end{thm}

\begin{proof}
The classes $v^{s}x_{1}^{i_{1}}x_{2}^{i_{2}}\ldots x_{m}^{i_{m}}
(\overline{x}_{{1}})^{j_{1}}(\overline{x}_{{2}})^{j_{2}}\ldots (\overline{x}_{{m}})^{j_{m}}$
generate $KU^*(BT^m)$ as a power series ring. The realification map $r$  is onto by
(\ref{eqn:bott.seq}). So, the classes $[I,J]^{(s)} \in KO^{-2s}(BT^m)$ generate
$KO^*(BT^m)$ as a $KO^*$-module. Relations (A) and (B) in
Theorem \ref{thm:ij.relations} imply that every element $[I,J]^{(s)} \in KO^{-2s}(BT^m)$ can be
written as  a linear combination of elements in $\mathcal{G}_2$. A product of two elements in 
$\mathcal{G}_2$ is not given explicitly in terms of elements of $\mathcal{G}_2$ by relation (C) but
repeated applications of relation (A) and (B) reduce the result of (C) to a linear combination
of elements of $\mathcal{G}_2$. \end{proof}

\begin{rem}
Lemma \ref{eqn:module.structure} and the proof of Theorem \ref{thm:second.set}
describe the $KO^*$-algebra structure of $KO^*(BT^m)$. In particular, as noted in Section
\ref{sec:introduction},  this result and Theorem \ref{thm:ko.prod.proj.basis.1} both
describe the completion of the representation ring $RO(T^m)$ at the augmentation
ideal.
\end{rem}

\begin{prop}
No finite relations exist among the elements of $\mathcal{G}_2$.
\end{prop}

\begin{proof}
For $m = 1$, a finite relation among elements of $\mathcal{G}_2$
would produces a relation of the form $r(p(x_1)) = 0$ where $p(x_1)$ is a polynomial.
The Bott sequence implies then that a formal power series $\theta(x_1)$ exists 
in $KU^0(CP^\infty)$ satisfying
$$\theta(x_1) - \theta(\overline{x}_1) = p(x_1).$$

\nd It is straightforward to check that no such relation can occur in $KU^0(CP^\infty)$. The case
$m >1$ reduces easily to the case $m = 1$. \end{proof}

\

On the other hand, infinite relations do occur in $KO^{-2s}(BT^m)$ among the generators $\mathcal{G}_2$.  
Consider a relation of the form
\begin{equation}\label{eqn:relation}
\sum_{k=0}^{\infty}a_{k}[I_k,J_k]^{(s)} = 0, \quad a_k \in \mathbb{Z},\quad  
[I_k,J_k]^{(s)} \in \mathcal{G}_2 \end{equation}

\nd When applied to (\ref{eqn:relation}), the complexification monomorphism produces a relation among 
the  generators $x_1,\ldots,x_m,\overline{x}_{1},\ldots,\overline{x}_{m}\;$   in $KU^{-2s}(BT^m)$     
which must then be a consequence of the relations
$$x_{i}\overline{x}_{i} + x_i + \overline{x}_{i} = 0 \quad i = 1, \ldots,m.$$

\nd For $m=1$, there is the example 

\begin{equation}\label{eqn:relation}
2[1,0] + \sum_{n = 2}^{\infty}(-1)^{n-1}[n,0]  = 0,\end{equation}

\nd which is just the realification map $r$  applied to the relation

$$\overline{x}_{1} =  - \frac{x_1}{1+x_1} =  - x_{1} + x_{1}^2 - x_{1}^3 + x_{1}^4 + \ldots$$

\nd in $KU^{0}(BT^1)$.

\begin{remark}
The two generating sets $\mathcal{G}_1$ and $\mathcal{G}_2$, 
described in Corollary \ref{cor:ko.prod.proj.basis.1} and (\ref{eqn:g2})
respectively, are distinguished as follows. Although in both cases, an infinite sum of allowable
generators will become finite under any restriction
\begin{equation}
KO^{*}(\prod_{i=1}^{m}{\mathbb{C}P^{\infty}}) \longrightarrow 
KO^{*}\big( \mathbb{C}P^{k_1} \times  \mathbb{C}P^{k_2} \times \cdots \times
\mathbb{C}P^{k_m}\big),
\end{equation}
relations of the type (\ref{eqn:relation}) in the generators $\mathcal{G}_2$ will produce
{\em finite\/} linear relations.
\end{remark}

\skp{0.3}

\section{The Davis-Januszkiewicz spaces}\label{sec:dj}
\subsection{The Davis-Januszkiewicz space associated to a simplicial complex}
\label{subsec:dj.intro} In Section
\ref{sec:introduction}, the Davis-Januszkiewicz space $\mathcal D\mathcal J(K_P)$, 
associated to simple polytope $P$, was defined in terms of a toric manifold $M^{2n}$.
More generally, a Davis-Januszkiewicz space  $\mathcal D\mathcal J(K)$ can be 
constructed for any simplicial complex $K$ by means of the generalized 
moment-angle complex construction 
$Z(K;(X,A))$ of
\cite{davis.jan}, \cite{buchstaber.panov.2}, \cite{denham.suciu} and \cite{bbcg}. 
A  description of  the space $\mathcal D\mathcal J(K)$ follows.

\nd \begin{defin}\label{dj.definition} 
Let $K$ be a simplicial complex with $m$ vertices. Identify simplices 
$\sigma \in K$ as increasing subsequences of $[m] = (1,2,3,\ldots,m)$.
The Davis-Januszkiewicz space  
$\mathcal D\mathcal J(K)$ is defined by
$$\mathcal D\mathcal J(K) = Z(K; (\mathbb{C}P^\infty, \ast)) \; \subseteq \; 
BT^m = \prod_{i=1}^{m}{\mathbb{C}P^{\infty}}$$

\nd where $\ast$ represents the basepoint and

$$Z(K; (\mathbb{C}P^\infty, \ast)) \;=\; \bigcup_{\sigma \in K}D(\sigma)$$

\nd with

\begin{equation}\label{eqn:esigma} 
D(\sigma) =\prod^m_{i=1}W_i \;,\quad {\rm where}\quad
W_i=\left\{\begin{array}{lcl}
\mathbb{C}P^\infty  &{\rm if} & i\in \sigma\\
\ast &{\rm if} & i\in [m]-\sigma.
\end{array}\right.\end{equation}
\end{defin}

\

A toric manifold $M^{2n}$ is specified by a simple $n$-dimensional polytope  and a 
{\em characteristic function\/} on its facets as described in  \cite{davis.jan}.
Equivalently,  $M^{2n}$ can be realized as a quotient. The characteristic function
corresponds to a specific choice of sub-torus $T^{m-n} \subseteq T^m$ which acts 
freely on the moment-angle complex $Z(K_P;(D^2,S^1))$ to give

$$M^{2n}\; \cong \; Z(K_P;(D^2,S^1))\big/T^{m-n}.$$

\

\nd This description of $M^{2n}$ yields an equivalence of Borel constructions

\begin{multline}\label{eqn:tmdj}
ET^m \times_{T^m}  Z(K_P;(D^2,S^1)) \\ \simeq  
ET^n \times_{T^n} \big(Z(K_P;(D^2,S^1))/T^{m-n}\big) \cong 
ET^n \times_{T^n} M^{2n} = \mathcal{DJ}(K_P).\end{multline}

\

\nd Moreover, for any simplicial complex $K$, there is an equivalence
(\cite{davis.jan}, \cite{buchstaber.panov.2} and \cite{denham.suciu}),

\begin{equation}\label{eqn:ds}
ET^m \times_{T^m}  Z(K;(D^2,S^1))  \; \cong \; 
Z(K; (\mathbb{C}P^\infty, \ast)).\end{equation}

\nd It follows that for $K = K_P$, the three descriptions of $\mathcal D\mathcal J(K_P)$
given by (\ref{eqn:tmdj}) and (\ref{eqn:ds}) agree up to homotopy equivalence


\

\subsection{The $KO^*$-rings of the Davis-Januszkiewicz spaces}\label{section:kodj}
It is well known that (\ref{eqn:dj.definition}) extends to $\mathcal D\mathcal J(K)$ and so,
for any complex-oriented cohomology theory $E^*$

\begin{equation}\label{eqn:dj.definition.2}
E^*(\mathcal D\mathcal J(K)) \;\cong\; E^*(BT^m)\big/I_{SR}^E \end{equation}

\

\nd where $I_{SR}^E$ is the Stanley-Reisner ideal described in Section
\ref{sec:introduction}. A related but more general result can be found in \cite{bbcg},
Theorem 2.35. Also from \cite{bbcg}, the following geometric results  will prove
useful for the computation of $KO^{*}(\mathcal D\mathcal J(K))$ in this section. Below, 
increasing subsequences of $[m] = (1,2,3,\ldots,m)$ are denoted by
$\sigma = (i_1,i_2,\ldots,i_k), \;\tau = (i_1,i_2,\ldots,i_t)$ and $\omega = (i_1,i_2,\ldots,i_s)$
and $X_{i_j} = \mathbb{C}P^\infty$ for all $i_j$.

\begin{thm}\label{thm:dj.split} The Davis-Januszkiewicz space  $\mathcal D\mathcal J(K)$ 
decomposes stably as follows. 

$$\Sigma\big(\mathcal D\mathcal J(K)\big)\;  \stackrel{\simeq}{\longrightarrow}\; 
\Sigma\big(\bigvee_{\sigma \in K} X_{i_1}\wedge X_{i_2} \wedge \ldots \wedge X_{i_k}\big).$$

\

\nd Moreover, there is a cofibration sequence 

$$\Sigma\big(\mathcal D\mathcal J(K)\big)
 \stackrel{i}{\longrightarrow}\
\Sigma\big(\bigvee_{\tau \in [m]} X_{i_1}\wedge X_{i_2} \wedge \ldots \wedge X_{i_t}\big)
 \stackrel{q}{\longrightarrow}\
\Sigma\big(\bigvee_{\omega \notin K} X_{i_1}\wedge X_{i_2} \wedge \ldots \wedge X_{i_s}\big),$$
\end{thm}

\nd where the map $i$ is split.

\

A particular case of (\ref{eqn:dj.definition.2}) is given by $E^* = KU^*$, so

\begin{equation}\label{eqn:ku.dj}
KU^*(\mathcal D\mathcal J(K)) \;\cong\; KU^*(BT^m)\big/I_{SR}^{KU}\end{equation}

\

\begin{remark}\label{rem:conjugates}
Notice that in the representation of $KU^{0}(BT^m)$ given in (\ref{eqn:conjugates}), 
the monomials generating the ideal $I_{SR}^{KU}$ could equally well contain a generator
$x_i$ or its conjugate $\overline{x}_i$.
\end{remark}

\

Theorem \ref{thm:dj.split} and the results of section \ref{sect:basic.calculation}
imply that $KO^*(\mathcal D\mathcal J(K))$ is concentrated in even degrees. The 
Bott sequence (\ref{eqn:bott.seq}) implies then that the realification map

$$r\colon KU^{*}(\mathcal{DJ}(K)) \longrightarrow KO^{*}(\mathcal{DJ}(K))$$

\

\nd is onto  and that the complexification map

$$c\colon KO^{*}(\mathcal{DJ}(K)) \longrightarrow KU^{*}(\mathcal{DJ}(K))$$

\

\nd is a monomorphism. The goal of the remainder of this section is to use the generators 
$\mathcal{G}_2$
of Section \ref{sec:ijays} to describe the ring  $KO^{*}(\mathcal{DJ}(K))$.  

\

Let $K$ be a simplicial complex on $m$ vertices. For $I = (i_1,i_2,\ldots,i_m)$ as in Section
\ref{sec:ijays}, set

$$\epsilon(I) = \{k : i_k \neq 0\} \subseteq [m]$$ 

\

\nd and let $SR_{KO}$ denote the ideal in $KO^{*}(BT^m)$ generated
by the set

\begin{equation}\label{eqn:ko.ideal}
\{[I,J]^{(s)} \in \mathcal{G}_2 : \epsilon(I) \cup \epsilon(J) \notin K\},\end{equation}

\

\nd where again, simplices of $K$ are identified as increasing subsequences of 
$[m] = (1,2,3,\ldots,m)$. The notation $SR_{KO}$ for the
$KO$ Stanley-Reisner ideal is more appropriate than
$I_{SR}^{KO}$ as it is structurally different from that for a 
complex-oriented theory. Next, the ideal $SR_{KO}$ is related to $r(I^{KU}_{SR})$.
The non-multiplicativity of the map $r$ makes necessary a preliminary lemma.

\begin{lem}\label{lem:ideal}
The abelian group $r(I^{KU}_{SR})$ is an ideal in $KO^*(BT^m)$.
\end{lem}

\begin{proof} With reference to the notation of Theorem \ref{thm:dj.split}, set 

$$\widehat{X}^\tau =  X_{i_1}\wedge X_{i_2} \wedge \ldots \wedge X_{i_t} \quad
\text{and} \quad \widehat{X}^\omega =  X_{i_1}\wedge X_{i_2} \wedge \ldots \wedge X_{i_s}.$$

\

\nd Recall here that each $X_{i_j} = \mathbb{C}P^\infty$. The split cofibration of 
Theorem \ref{thm:dj.split} gives rise to the diagram following

\begin{equation}\label{eqn:dj.cd}
\begin{CD}
\widetilde{KU}^{-2s}\big(\mathcal{DJ}(K)\big) @<{i^{*}}<<
\widetilde{KU}^{-2s}\big(\bigvee\limits_{\tau \in [m]}\widehat{X}^\tau\big) @<{q^{*}}<<  
\widetilde{KU}^{-2s}\big(\bigvee\limits_{\omega \notin K}\widehat{X}^\omega\big)\\
\\
@VV{r}V              @VV{r}V            @VV{r}V \\
\\
\widetilde{KO}^{-2s}\big(\mathcal{DJ}(K)\big) @<{i^{*}}<<
\widetilde{KO}^{-2s}\big(\bigvee\limits_{\tau \in [m]}\widehat{X}^\tau\big) @<{q^{*}}<<  
\widetilde{KO}^{-2s}\big(\bigvee\limits_{\omega \notin K}\widehat{X}^\omega\big).\\
\end{CD}
\end{equation}

\

\nd The maps $i^*$ are onto and so 
$\widetilde{KO}^{*}\big(\mathcal{DJ}(K)\big)$ is a quotient of $\widetilde{KO}^{*}(BT^{m})$. A diagram chase is
needed next. Let $x \in \widetilde{KO}^{-2s}\big(\bigvee\limits_{\tau \in [m]}\widehat{X}^\tau\big)$ be such
that $i^*(x) = 0$.
Then, $y \in \widetilde{KO}^{-2s}\big(\bigvee\limits_{\omega \notin K}\widehat{X}^\omega\big)$ exists
satisfying $q^*(y) = x$. Since $r$ is onto, 
$z \in \widetilde{KU}^{-2s}\big(\bigvee\limits_{\omega \notin K}\widehat{X}^\omega\big)$ exists with
$r(z) = y.$ Then

$$r(q^*(z)) = q^*(r(z)) = x.$$

\nd Now $q^*(z) \in I_{SR}^{KU}$ which implies that $x \in r(I_{SR}^{KU})$. 
Conversely, the commutativity of the left-hand half of (\ref{eqn:dj.cd}) implies
that if $x \in r(I_{SR}^{KU})$ then $i^*(x) = 0$. Hence $r(I_{SR}^{KU})$ is the
kernel of the map
\begin{equation}\label{eqn:kernel}
i^* \colon \widetilde{KO}^*(BT^m) \longrightarrow \widetilde{KO}^{*}(\mathcal{DJ}(K)).
\end{equation}

\nd In particular,  $r(I^{KU}_{SR})$ is an ideal in $KO^*(BT^m)$ as required.
\end{proof}

\

The next proposition  allows a characterization of this important ideal in terms of the condition 
(\ref{eqn:ko.ideal}).

\begin{prop}\label{prop:ideals.equal}
As ideals in $KO^*(BT^m)$ 
$$ r(I_{SR}^{KU}) \; = \; SR_{KO}.$$
\end{prop}

\begin{proof}
Let $[I,J]^{(s)} \in SR_{KO}$. Since $\epsilon(I) \cup \epsilon(J) \notin K$,
\begin{equation}\label{eqn:monomial}
[I,J]^{(s)}\; =\;  r\big(v^{s}y_{\alpha_1}y_{\alpha_2}\ldots y_{\alpha_k} \mathfrak{m}\big)
\end{equation}

\nd where, in the light of Remark \ref{rem:conjugates}, each $y_{\alpha_j} = x_{\alpha_j}$ or $\overline{x}_{{\alpha_j}}\,$, $\;\{\alpha_1,\alpha_2,\ldots,\alpha_k\} \notin K$
and $\mathfrak{m}$ is a monomial in the classes 
$x_1,\ldots,x_m, \overline{x}_{1},\ldots,\overline{x}_{m}$. (Notice here that the choice of 
${\alpha_1},{\alpha_2}, \ldots, {\alpha_k}$ in (\ref{eqn:monomial}) may not be unique.) 
Now $v^{s}y_{\alpha_1}y_{\alpha_2}\ldots y_{\alpha_k} \mathfrak{m} \in I_{SR}^{KU}$ and so 
$[I,J]^{(s)} \in \; r(I_{SR}^{KU}).$ Conversely,  an element in 
$ r(I_{SR}^{KU})$ is a $KO^*$-sum of elements each of the form
$r\big(v^{s}y_{\alpha_1}y_{\alpha_2}\ldots y_{\alpha_k} \mathfrak{n}\big)$ again with
each $y_{\alpha_j} = x_{\alpha_j}$ or $\overline{x}_{{\alpha_j}}\,$, $\;\{\alpha_1,\alpha_2,\ldots,
\alpha_k\} \notin K$ and $\mathfrak{n}$ is a monomial in the classes 
$x_1,\ldots,x_m, \overline{x}_{1},\ldots,\overline{x}_{m}$. Now 
$r\big(v^{s}y_{\alpha_1}y_{\alpha_2}\ldots y_{\alpha_k} \mathfrak{n}\big) = [I',J']^{(s)}$ for some $I'$ 
and $J'$ and moreover,  $\epsilon(I') \cup \epsilon(J') \notin K$ because $\;\{\alpha_1,\alpha_2,\ldots,\alpha_k\} 
\notin K$. It follows that $r(I_{SR}^{KU}) \subset SR_{KO}$, proving the converse.
\end{proof}

\nd The main theorem of this section follows.

\begin{thm}\label{thm:ko.of.dj}
There is an isomorphism of graded rings

$$KO^{*}(\mathcal{DJ}(K))\; \cong\; KO^{*}(BT^{m})\big/SR_{KO}$$ 
\end{thm}

\begin{proof}
Proposition \ref{prop:ideals.equal} identifies $SR_{KO}$ as $r(I_{SR}^{KU})$, which 
is the kernel of the map $i^*$ of (\ref{eqn:kernel}) in the proof of Lemma \ref{lem:ideal}.
\end{proof}

\

The examples following illustrate calculations in $KO^{0}(\mathcal{DJ}(K))$ based on 
Theorem \ref{thm:ko.of.dj}. The relations of Theorem \ref{thm:ij.relations} are 
used with $s = t = 0$ and the elements $[I,J]$ are to be interpreted modulo the ideal of
relations $SR_{KO}$.

\begin{examples}\label{exm:two.points}\

\nd (1) Let $K = \big\{\{v_1\},\{v_2\}\big\}$ be the simplicial complex consisting of two
distinct vertices. Classes of the form $[(i,0),(0,0)]$ and $[(0,h),(0,0)]$  represent 
$\mathcal{G}_2$ generators of 
$KO^0(\mathbb{C}P^{\infty} \times \ast)$ 
and $KO^0(\ast \times \mathbb{C}P^{\infty})$
respectively in  $KO^0(BT^2)$ as described by (\ref{eqn:g2}). Now, for $i$ and $h$ not both zero,
\begin{align*}
[(i,0),(0,0)]\cdot [(0,h),(0,0)] &= [(i,h),(0,0)] +  [(0,h),(i,0)]\\
&= 0\quad \text{by}\;(\ref{eqn:ko.ideal})\\
\end{align*}
\skp{-0.8}
\nd which is consistent with the fact that $\mathcal{DJ}(K) = 
\mathbb{C}P^{\infty} \vee \mathbb{C}P^{\infty}$ in this case.

\

\nd (2) Let $L = \big\{\{v_1\},\{v_2\},\{v_3\},\{v_4\},\{v_1,v_2\},\{v_2,v_3\},\{v_3,v_4\} ,\{v_2,v_4\},
\{v_2,v_3,v_4\}\big\}$ be the simplicial complex consisting of a $1$-simplex wedged to a 
$2$-simplex at the vertex $v_2$. Here, classes of the form $[(i_1,i_2,0,0),(0,0,0,0)]$ and 
$[(0,h_2,h_3,0),(0,0,0,0)]$ represent $\mathcal{G}_2$ generators of 
$KO^0(\mathbb{C}P^{\infty} \times \mathbb{C}P^{\infty} \times \ast \times \ast)$ 
and $KO^0(\ast \times \mathbb{C}P^{\infty} \times \mathbb{C}P^{\infty} \times \ast)$
respectively in  $KO^0(BT^4)$.  Now
\begin{multline*}
[(i_1,i_2,0,0),(0,0,0,0)] \cdot [(0,h_2,h,_3,0),(0,0,0,0)] \\
=  [(i_1,i_2 + h_2,h_3,0),(0,0,0,0)] + [(0,h_2,h_3,0),(i_1,i_2,0,0)]
= 0\; \; \text{by}\;(\ref{eqn:ko.ideal})\\
\end{multline*}
\skp{-0.8}
\nd reflecting the fact that $\{v_1,v_2,v_3\} \notin L$. Moreover
\begin{multline}\label{eqn,calculation}
[(i_1,i_2,0,0),(0,0,0,0)] \cdot [(l_1,l,_2,0,0),(0,0,0,0)] \\
=  [(i_1+l_1,i_2 + l_2,0,0),(0,0,0,0)] + [(l_1,l_2,0,0),(i_1,i_2,0,0)].
\end{multline}

\nd Repeated application of relations (A) and (B) in Theorem \ref{thm:ij.relations}
reduce the right hand side of (\ref{eqn,calculation}) to a sum of terms of the
form $[(\ast,\ast,0,0), (\ast,\ast,0,0)]$ each of which satisfies the $\mathcal{G}_2$
condition for $KO^0(\mathbb{C}P^{\infty} \times \mathbb{C}P^{\infty})$. This is 
consistent with the fact that $KO^*(\mathbb{C}P^{\infty} \times \mathbb{C}P^{\infty})$
is a $KO^*$-subalgebra of $KO^{*}(\mathcal{DJ}(K))$ corresponding to the simplex 
$\{v_1,v_2\} \in L$.
\end{examples}

\

\subsection{The {\sc cat}$(K)$ approach.}
Definition \ref{dj.definition} expresses  $\mathcal{DJ}(K)$ as the colimit of an
exponential diagram $BT^K$ (\cite{panov.ray.vogt}, where $D(\sigma)$ is written $BT^{\sigma}$),
over the category  {\sc cat}$(K)$ associated to the posets of faces of $K$. 
Since $BT^K$ is a cofibrant diagram, its homotopy colimit is homotopy equivalent to 
$\mathcal{DJ}(K)$ also.
The $KO^*$ version of the 
Bousfield-Kan spectral sequence, studied in \cite[Section 3]{notbohm.ray}, applies in
this case and gives an alternative calculation of $KO^{*}(\mathcal{DJ}(K))$ in terms of the 
{\sc cat}$^{op}(K)$-diagram $KO^*(BT^K)$ whose value on each face 
$\sigma \in K$ is $KO^*(D(\sigma))$. The arguments of \cite{notbohm.ray} apply unchanged 
and are similar to those of Section \ref{section:kodj}. They imply that
the spectral sequence collapses at the $E_2$-term and is concentrated entirely along the 
vertical axis. So the edge homomorphism gives an isomorphism

\begin{equation}\label{eqn:edge}
KO^*(\mathcal{DJ}(K)) \stackrel{\cong}{\longrightarrow}  \;\lim KO^*(BT^K)
\end{equation}

\

\nd of $KO^*$-algebras, by analogy with \cite[Corollary 3.12]{notbohm.ray}.
\skp{0.1}
Informally, the elements of $\;\lim KO^*(BT^K)$ are considered as  finite sequences $(u_\sigma)$ whose
terms $u_\sigma \in KO^*(BT^\sigma)$ are compatible under the inclusions 
$i\colon BT^\sigma \longrightarrow BT^\tau$ for every $\tau \supset \sigma$. More precisely,
the isomorphism (\ref{eqn:edge}) leads to the conclusion following.

\begin{thm}\label{thm:bk.description}
As $KO^*$-algebras, $KO^*(\mathcal{DJ}(K))$ is isomorphic to 
$$\big\{(u_\sigma) \in \prod_{\sigma \in K} KO^*(BT^\sigma) \colon i^*(u_\tau) = u_\sigma 
\;\; \text{for every}\;\; \tau \supset \sigma\big\} $$

\nd where the multiplication and $KO^*$-module structure are defined termwise. \hfill $\square$
\end{thm}

\nd Theorem \ref{thm:bk.description} extends to $E^*(\mathcal{DJ}(K))$ for any arbitrary
cohomology theory. The corollary following is complementary to Theorem \ref{thm:ko.of.dj}.

\begin{cor}\label{cor:ell}
The natural homomorphism
$$\ell \colon KO^*(BT^m) \longrightarrow \lim KO^*(BT^K)$$

\nd is onto with kernel equal to the ideal $SR_{KO}$ of Theorem \ref{thm:ko.of.dj}.
\end{cor}

\begin{proof}
The homomorphism $\ell$ is induced by the projections 
$KO^*(BT^m) \rightarrow KO^*(BT^\sigma)$ as $\sigma$ ranges over the faces of $K$, hence it
is onto.

Theorem \ref{thm:second.set} describes each summand $KO^*(BT^\sigma)$ of $KO^*(BT^m)$
as generated over $KO^*$ by those elements $[I,J]^{(s)}$ of $\mathcal{G}_2$ for which
$\epsilon(I) \cup \epsilon(J) \subseteq \sigma$. Moreover, Theorem \ref{thm:bk.description}
implies that $\ell([I,J]^{(s)}) = 0$ if and only if $[I,J]^{(s)}$ satisfies 
$\epsilon(I) \cup \epsilon(J) \notin K$ as in (\ref{eqn:ko.ideal}). So, $\ell$ maps non-trivial
formal sums of elements in $\mathcal{G}_2$ to zero if and only if they lie in $SR_{KO}$.
\end{proof}

Corollary \ref{cor:ell} generalizes to an arbitrary cohomology theory and establishes an
isomorphism
$$E^*(BT^m)\big/\text{ker}\;\ell \longrightarrow E^*(\mathcal{DJ}(K))$$

\nd of $E^*$-algebras.

\

It is instructive to revisit Examples \ref{exm:two.points}  from this 
complementary viewpoint.

\begin{examples}\

\nd (1) If $K = \{\{v_1\},\{v_2\}\}$, then {\sc cat}$(K)$ contains the $(-1)$-simplex $\emptyset$ and two
 $0$-simplices.  Theorem \ref{thm:bk.description} gives $KO^*(\mathcal{DJ}(K))$ as the
 $KO^*$-algebra  $KO^*(BT^{\{v_1\}} \oplus KO^*(BT^{\{v_2\}}$. The homomorphism $\ell$ of  
 Corollary \ref{cor:ell}  maps the elements $[(i,0),(0,0)]$ and  $[(0,h),(0,0)]$ of $KO^0(BT^2)$ to the
 elements $\big([(i),(0)],0\big)$ and  $\big(0,[(h),(0)]\big)$; in particular, their product is zero.
 
 \
 
\nd (2) If $L = \big\{\{v_1\},\{v_2\},\{v_3\},\{v_4\},\{v_1,v_2\},\{v_2,v_3\},\{v_3,v_4\} ,\{v_2,v_4\},
\{v_2,v_3,v_4\}\big\}$, then  {\sc cat}$(L)$ contains the $(-1)$-simplex $\emptyset$, four
$0$-simplices, four 1-simplices and one 2-simplex.  Theorem \ref{thm:bk.description}  expresses $KO^*(\mathcal{DJ}(K))$ as a certain  $KO^*$-subalgebra of 
 
$$KO^*(BT^{\{v_2\}}) \times KO^*(BT^{\{v_1,v_2\}}) \times KO^*(BT^{\{v_2,v_3,v_4\}}),$$
\skp{0.1}
\nd which may be identified as the pullback
\begin{equation}\label{eqn:pullback}
KO^{*}(BT^{\{v_1,v_2\}})  \oplus_{KO^{*}(BT^{\{v_2\}})} KO^{*}(BT^{\{v_2,v_3,v_4\}}).
\end{equation}

\nd The elements of (\ref{eqn:pullback}) consist of ordered pairs $(u,w)$, for which
$u \in KO^*(BT^{\{v_1,v_2\}})$ and $w \in KO^*(BT^{\{v_2,v_3,v_4\}})$ share a common restriction to 
$KO^*(BT^{\{v_2\}})$. Pairs are multiplied coordinate-wise; products of the form $(u,0)\cdot(0,w)$
give $(0,0) = 0$. For $i_1$ and $h_3$ nonzero, the homomorphism $\ell$ of
Corollary \ref{cor:ell} maps the elements 
$$[(i_1,i_2,0,0),(0,0,0,0)] \;\; \text{and} \;\;  [(0,h_2,h_3,0),(0,0,0,0)]$$

\nd of $KO^0(BT^4)$ to the pairs
$$\big([(i_1,i_2),(0,0)],0\big) \;\; \text{and}\;\;  \big(0,[(h_2,h_3,0),(0,0,0)]\big)$$ 

\nd respectively. Their product is zero as required. Similarly, $\ell$ maps
$[(i_1,i_2,0,0),(0,0,0,0)]$  and  $[(l_1,l,_2,0,0),(0,0,0,0)]$ to the pairs
$$\big([(i_1,i_2),(0,0)],0\big) \; \text{and}\; \big([(l_1,l_2),(0,0)],0\big)$$

\nd respectively and, their product is $\big([(i_1+l_1, i_2+l_2),(0,0)]  + [(l_1,l_2),(i_1,i_2)],0\big)$.
\end{examples}

\skp{0.4}

\section{Toric manifolds }\label{sec:tm}
\subsection{Background.} Briefly, a toric manifold $M^{2n}$ is a manifold covered by local 
charts $\mathbb{C}^n$, each with the standard $T^n$ action, compatible in such a way that 
the quotient $M^{2n}\big/T^n$ has the structure of a simple polytope \pn. A {\em simple} 
polytope $P^n$ has the property that at each vertex, exactly $n$ facets  intersect.
Under the $T^n$ action, each copy of $\mathbb{C}^n$ must project to an $\mathbb{R}^n_+$  
neighborhood of a vertex of \pn.
The construction of Davis and Januszkiewicz (\cite{davis.jan}, section $1.5$) realizes all
such manifolds as follows. Let

$${\mathcal F} = \{F_{1},F_{2},\ldots,F_{m}\}$$

\

\nd denote the set of facets of \pn. The fact that \pn $\;$ is 
simple implies that every codimension-$l$ face $F$ can be 
written uniquely as

$$F = F_{i_{1}} \cap F_{i_{2}} \cap \cdots \cap F_{i_{l}}$$

\

\nd where the $F_{i_{j}}$ are the facets containing $F$. Let

$$\lambda : {\mathcal F} \longrightarrow \mathbb{Z}^n$$

\

\nd be a function into an $n$-dimensional integer lattice satisfying the 
condition that whenever 
$F = F_{i_{1}} \cap F_{i_{2}} \cap \cdots \cap F_{i_{l}}$ then 
$\{\lambda(F_{i_{1}}),\lambda(F_{i_{2}}), \ldots ,\lambda(F_{i_{l}})\}$ span 
an $l$-dimensional submodule of $\mathbb{Z}^n$ which is a direct summand.
Next, regarding $\mathbb{R}^n$ as the Lie algebra of $T^n$, 
$\lambda$ associates to each codimension-$l$ face $F$ of \pn \/ a rank-$l$
subgroup $G_F \subset T^n$. Finally, let $p \in$ \pn $\;$ and $F(p)$ be 
the unique face with $p$ in its relative interior. Define an equivalence
relation $\sim$ on $T^n$ $\times$ \pn $\;$ by $(g,p) \sim (h,q)$ if and only
if $p = q$ and $g^{-1}h \in G_{F(p)} \cong T^l$. Then

$$M^{2n} \cong M^{2n}(\lambda) = T^n \times P^n\big/\!\sim$$ 

\

\nd and, $M^{2n}$ is a smooth, closed, connected, $2n$-dimensional manifold with 
$T^n$ action induced by left translation (\cite{davis.jan}, page 423). The 
projection $\pi \colon M^{2n} \rightarrow P^n$ is induced from the projection
$T^n \times$ \pn $\rightarrow$ \pn. It is noted in \cite{davis.jan} that every smooth projective 
toric variety has this description. 

The goal of this section is an analogue of (\ref{eqn:cohomology})
for the $KO^*$-rings of certain toric manifolds $M^{2n}$. For toric manifolds determined by a 
simple polytope and a characteristic map on its facets, a description of the $KO^*$-module
structure of $KO^*(M^{2n})$ was given in \cite{bb} in terms of $H^*(M^{2n}; \mathbb{Z}_2)$ 
as a module over $\sq^1$ and $\sq^2$. A more refined computation of the $KO^*$-module
structure, for certain families of manifolds $M^{2n}$, is presented in \cite{nishimura}. The
$KO^*$-ring structure for families of toric manifolds known as Bott towers may be found
in \cite{civan.ray}, without reference to $KO^*(\mathcal{DJ}(K))$.

\

\subsection{The Steenrod algebra structure of toric manifolds.} As in
Section \ref{sect:basic.calculation}, let
$\mathcal{A}_1$ denote the subalgebra of the Steenrod algebra generated by
$\sq^1$ and $\sq^2$.  Let $S^{0}$ denote the $\mathcal{A}_1$-module consisting of a 
single class in dimension 0 and the trivial action of $\sq^1$ and $\sq^2$.
Denote by $\mathcal{M}$ the $\mathcal{A}_1$-module with a class $x$ in dimension $0$, 
a class $y$ in dimension $2$ and the action given by $\sq^2(x) = y$.

According to (\ref{eqn:cohomology}), $H^*(M^{2n}; \mathbb{Z}_2)$ is concentrated in even
degree and so, as an $\mathcal{A}_1$-module, must be isomorphic to a direct sum of
suspended copies of the modules $S^0$ and $\mathcal{M}$. That is, there is a 
decomposition

\begin{equation}\label{eqn:a.decomposition}
H^*(M^{2n}; \mathbb{Z}_2) \;\cong\;  \bigoplus_{i=0}^{n}s_{i}\Sigma^{2i}S^0\;\oplus\;
\bigoplus_{j=0}^{n-1}m_{j}\Sigma^{2j}\mathcal{M},
\qquad s_i, m_j \in \mathbb{Z}.
\end{equation}

\nd The numbers $s_i$ and $m_j$ were labelled ``BB-numbers'' in \cite[Section 5]{civan.ray}.
The $\sq^2$-homology of $M^{2n}$, $H^*(M^{2n};\sq^2)$, is zero precisely when 
$s_j = 0$ for all $j$.

\

\begin{examples}
The toric manifolds $\mathbb{C}P^{2k}$ are $\sq^2$-acyclic for any positive integer $k$.
\end{examples}
\begin{examples}\label{exm:bott.towers}
The toric manifolds $\mathbb{C}P^{2k+1}$ have $s_i = 0$ for $i \leq k$ and $s_{k+1} =1$,
for any positive integer $k$. The {\em terminally odd\/} Bott towers of \cite[Section 5]{civan.ray}
have $s_1 = 1$ and $s_i = 0$ for $i\geq 2$; the {\em totally even\/} towers have $m_j = 0$
for every $j$.
\end{examples}
\begin{examples}
The non-singular toric varieties $X^n(r;a_r,\ldots,a_n)$ constructed in \cite{nishimura}
and satisfying $2\leq r \leq n$, $a_j\in  \mathbb{Z}$ and $n-r$ odd are all $\sq^2$-acyclic
These varieties correspond to $n$-dimensional fans having $n+2$ rays .
\end{examples}

\begin{rem}
The preprint \cite{bbcg3} contains a construction of families of toric manifolds derived
from a given one. Work is in progress to confirm that this construction can be done in 
such a way that the family of derived toric manifolds will each be $\sq^2$-acyclic, though
this property might not be satisfied by the initial one. 
\end{rem}

The next proposition is an immediate consequence of the calculation in \cite{bb}.

\begin{prop}\label{prop:even.ko}
If  $M^{2n}$ is $\sq^2$-acyclic, then the graded ring $KO^*(M^{2n})$
is concentrated in even degree and has no additive torsion. \hfill $\square$
\end{prop}

\subsection{The $KO$-rings of $\sq^2$-acyclic toric manifolds.}\label{sec:kotm}

Recall from (\ref{eqn:basic.fibration}) the Borel fibration for toric manifolds,
\begin{equation}\label{eqn:fibration}
M^{2n} \stackrel{i}{\longrightarrow}  ET^n \times_{T^n} M^{2n} \stackrel{p}{\longrightarrow} BT^n
\end{equation}
\nd with total space  $\mathcal{DJ}(K)$. 

\begin{thm}\label{thm:ko.of.m}
For any $\sq^2$-acyclic toric manifold $M^{2n}$, there is an isomorphism 

$$KO^{*}(M^{2n}) \; \cong \;  KO^{*}\big(\mathcal{DJ}(K)\big)\Big/r(J^{KU})$$

\nd of $KO^*$-algebras, where  $r$ is the realification map and $J^{KU}$ is the 
ideal  defined in (\ref{eqn:cohomology}).
\end{thm}

\begin{remark}
Notice  that $r(J^{KU})$, which
is the realification of the ideal generated by the image of 
$ KU^{*}(BT^{n})  \xrightarrow{p^{*}} KU^{*}\big(\mathcal{DJ}(K)\big)$,
is not the same as $J^{KO}$ which is the ideal generated by
$p^{*}\big(KO^{*}(BT^{n})\big)$; this represents
a significant departure from the situation for complex-oriented $E^*(M^{2n})$.
As in Lemma \ref{lem:ideal}, the non-multiplicativity of the map $r$ implies that $r(J^{KU})$
is not in general an ideal but Theorem \ref{thm:ko.of.m} confirms that 
$KO^{*}\big(\mathcal{DJ}(K)\big)\Big/r(J^{KU})$ is multiplicatively closed.
\end{remark}

\

\nd {\em Proof of Theorem \ref{thm:ko.of.m}.\/} 
The Bott sequences (\ref{eqn:bott.seq}) for $M^{2n}$, $\mathcal{DJ}(K)$ and $BT^n$ 
link together to give the commutative diagram.

\begin{equation}\label{eqn:manifold.diagram}
\begin{CD}
KO^{*-2}(M^{2n}) @<{i^*}<<
KO^{*-2}\big(\mathcal{DJ}(K)\big) @<{p^{*}}<<  KO^{*-2}(BT^{n})\\
\\
@VV{\chi}V              @V{\chi}VV            @V{\chi}VV             \\
\\
KU^{*}(M^{2n}) @<{i_{KU}^*\; \text{onto}}<<
KU^{*}\big(\mathcal{DJ}(K)\big) @<{p^{*}}<<  KU^{*}(BT^{n})\\
\\
@VV{r\; \text{onto}}V              @VV{r\; \text{onto}}V            @VV{r\; \text{onto}}V \\
\\
KO^{*}(M^{2n}) @<{i^*}<<
KO^{*}\big(\mathcal{DJ}(K)\big) @<{p^{*}}<<  KO^{*}(BT^{n})
\end{CD}
\end{equation}

\skp{0.5}

\nd Recall now that Proposition \ref{prop:even.ko} implies that  $KO^{*}(M^{2n})$ 
is concentrated in even degrees and so all the Bott sequences are short exact.
The lower left commutative square in (\ref{eqn:manifold.diagram}) implies that the
maps $i^*$ are onto. A diagram chase is needed next to identify the kernel of $i^*$.

Let $z \in KO^{*}\big(\mathcal{DJ}(K)\big)$ and suppose
$i^*(z) = 0$. Since $r$ is onto, $y \in KU^{*}\big(\mathcal{DJ}(K)\big)$ exists
with $r(y) = z$. Then $r(i_{KU}^*(y)) = i^*(z) = 0$. The exactness of the leftmost
Bott sequence implies now that $x \in KO^{*-2}(M^{2n})$ exists with 
$\chi(x) = i_{KU}^*(y)$. The map $i^*$ is onto so $w \in KO^{*-2}\big(\mathcal{DJ}(K)\big)$
satisfying $i^*(w) = x$. Then

$$i_{KU}^*(y - \chi(w)) = i_{KU}^*(y) - i^*(\chi(w)) =  i_{KU}^*(y) - \chi(i^*(w)) = 
i_{KU}^*(y) - \chi(x) = 0.$$

\

\nd So $y - \chi(w) \in \big\langle p^{*}\big(KU^{*}(BT^{n})\big)\big\rangle$ by
(\ref{eqn:cohomology}) for $E = KU$. Finally, $r(y - \chi(w)) = r(y) = z$ and so
$z \in {r\big(\big\langle p^{*}(KU^{*}(BT^{n}))\big\rangle\big)}$ as required. \hfill 
$\square$

\

\subsection{Further examples.}
A few simple examples illustrate the fact that the situation  is considerably more difficult
when $M^{2n}$ is not $\sq^2$-acyclic. In all which follows, the number $s_i$ and $m_j$
are those defined by (\ref{eqn:a.decomposition}).

Manifolds $M^{2n}$  for which all $m_j = 0$, as is the case for the totally even
Bott towers of Examples \ref{exm:bott.towers}, have $KO^*(M^{2n})$ a free 
$KO^*$-module. Particularly revealing is the most basic case
$M^{2n} = \prod_{k=1}^{n} \mathbb{C}P^1$ with $n=1$. Recall from Section 
\ref{sec:ijays} that
$$KO^{*} \;\cong \;  \mathbb{Z}[e , \alpha, \beta, \beta^{-1}]/(2e, e^3, e\alpha, 
\alpha^{2} - 4\beta)$$

\nd with $e \in KO^{-1}, \alpha \in KO^{-4}$ and $\beta \in KO^{-8}$.

\begin{exm}\label{exm:cp1}
The classes $X_1^{(s)} \in KO^{-2s}(\mathbb{C}P^\infty)$ and $X_i^{(0)} = X_i$,
specified in Definition \ref{defin:don.generator.notation}, restrict  to classes in 
$KO^{-2s}(\mathbb{C}P^1)$ which also will  be denoted by $X_1^{(s)}$ and $X_1$.
The $KO^*$-algebra $KO^{*}(\mathbb{C}P^1)$ is isomorphic to $KO^*[g]\big/(g^2)$
where $g \in \widetilde{KO}^2(\mathbb{C}P^1)$ is the generator arising from the unit
of the spectrum $KO$. In particular 
$$e^2g = X_1 \in KO^{(0)}(\mathbb{C}P^1) \quad \text{and} \quad 2\beta{g} = X_1^{(3)} \in
KO^{-6}(\mathbb{C}P^1).$$  Now $\mathbb{C}P^1$ is the smooth toric variety
associated to the simplicial complex $K = \big\{\{v_1\}, \{v_2\}\big\}$ in the manner
described in Section \ref{sec:introduction}. The fibration (\ref{eqn:fibration}) specializes to

$$\mathbb{C}P^1 \stackrel{i}{\longrightarrow} S(\eta \oplus \mathbb{R}) 
\stackrel{p}{\longrightarrow} BT^1,$$

\

\nd the total space of which is the  sphere bundle of $\eta \oplus \mathbb{R}$. So
$\mathcal{DJ}(K)$ is homotopy equivalent to $\mathbb{C}P^\infty \vee \mathbb{C}P^\infty$.
(Of course, this agrees with the description given by (\ref{eqn:ds}).) The map $i$ includes
$\mathbb{C}P^1$ into each wedge summand by pinching the equator.

It follows from \cite[Section 4]{civan.ray} that 
\begin{enumerate}
\item $i^*$ is an epimorphism onto
$KO^d(\mathbb{C}P^1)$ for all $d \equiv \hspace{-0.16in}/ \;1,2$ mod $8$. 
\item If $d = 1-8t$ then
$e\beta^{t}g$ has order $2$ but  $KO^{d}(\mathbb{C}P^\infty \vee \mathbb{C}P^\infty) = 0$.
\item  If $d = 2-8t$ then $2\beta^t{g} \in \text{Im}(i^*)$  but $\beta^t{g} \notin \text{Im}(i^*)$.
\end{enumerate}

\nd These details combined with diagram (\ref{eqn:manifold.diagram}) confirm that 
$$\text{Im}{(i^*)} \; \cong \;KO^{*}\big(\mathcal{DJ}(K)\big)\Big/r(J^{KU})$$

\nd in dimensions $\equiv \hspace{-0.16in}/ \;3,4$ mod $8$. 
\end{exm}

Example \ref{exm:cp1} extends to an analysis of various toric manifolds with a 
single non-zero $s_i$ but unrestricted $m_j$.

\begin{exm}
The projective space $\mathbb{C}P^{4k+1}$ has $s_{2k+1} = 1$ and all other $s_i = 0$.
It  is the smooth toric variety associated to the simplicial complex $K$ which is the boundary
of the simplex $\Delta^{4k+1}$.

The $KO^*$-algebra $KO^*(\mathbb{C}P^{4k+1})$  admits $KO^*(S^{8k+2})$ as an
additive summand, generated by $h \in KO^{8k+2}(\mathbb{C}P^{4k+1})$ such that
$h^2 = 0$. In particular,
$$e^2\beta^k{h} = X_{1}^{2k+1} \in KO^0(\mathbb{C}P^{4k+1}) \quad \text{and}  \quad
2\beta^{k+1}h = X_{1}^{2k}X_1^{(3)} \in KO^{-6}(\mathbb{C}P^{4k+1}).$$

\nd It follows from Example \ref{exm:cp1} that 

$$i^*\colon KO^d(\big(\mathcal{DJ}(K)\big) \longrightarrow KO^d(\mathbb{C}P^{4k+1})$$

\nd is an epimorphism for $d \equiv \hspace{-0.16in}/ \;1,2$ mod $8$. So the cokernel
of $i^*$ is isomorphic to the $\mathbb{Z}/2$ vector space generated by the elements
$e\beta^t{h}$ and $\beta^t{h}$, whereas 

$$\text{Im}{(i^*)} \; \cong \;KO^{*}\big(\mathcal{DJ}(K)\big)\Big/r(J^{KU})$$

\nd in dimensions $\equiv \hspace{-0.16in}/ \;3,4$ mod $8$. 
 \end{exm}
 
 \begin{exm}
 A terminally odd Bott tower $M^{2n}$ has $s_1 = 1$ and all other $s_i = 0$. In this case the
 simplicial complex $K$ is the boundary of an $n$-dimensional cross-polytope. As in
 Example \ref{exm:cp1}, it follows that the cokernel of $i^*$ is isomorphic to the 
 $\mathbb{Z}_2$-vector space generated by the elements $e\beta^t{g}$ and $\beta^t{g}$
 whereas 

$$\text{Im}{(i^*)} \; \cong \;KO^{*}\big(\mathcal{DJ}(K)\big)\Big/r(J^{KU})$$

\nd in dimensions $\equiv \hspace{-0.16in}/ \;3,4$ mod $8$. 
 \end{exm}
 
\
 
 \begin{ack}
 The authors are grateful to Matthias Franz for stimulating and helpful conversations
 about moment-angle complexes, Davis-Januszkiewicz spaces and toric manifolds.
 \end{ack}
 
 \newpage
 
\bibliographystyle{amsalpha}

\begin{thebibliography}{99}


\bibitem{jfa} J.~F.~Adams, {\em Stable Homotopy and Generalized Homology\/}, 
University of Chicago Press, 1974.

\bibitem{anderson} D.~W.~Anderson, {\em The Real $K$-theory of Classifying Spaces\/},
PNAS 1964 {\bf 51}: 634--636.

\bibitem{atiyah.segal} M.~F.~Atiyah and G.~B.~Segal {\em Equivariant $K$-theory and 
Completion\/}, J. Differential Geometry {\bf 3} (1969): 1--18.


\bibitem{bb}  A.~Bahri, M.~Bendersky, {\em The $KO$-theory of toric manifolds\/},
Trans. AMS, 2000 {\bf 352}, no. 3, 1191--1202

\bibitem{bbcgpnas}A.~Bahri, M.~Bendersky, F.~R.~Cohen, and S.~Gitler,
{\em Decompositions of the polyhedral product functor with applications to
moment-angle complexes and related spaces}, PNAS 2009 {\bf 106}:12241--12244.

\bibitem{bbcg}A.~Bahri, M.~Bendersky, F.~R.~Cohen, and S.~Gitler,  {\em The
polyhedral product functor: A method of decomposition for moment-angle complexes,
arrangements and related spaces}. To appear in Advances in Mathematics.

\bibitem{bbcg3} A.~Bahri, M.~ Bendersky, F.~ Cohen and S.~ Gitler, {\em On an 
Infinite Family of Toric Manifolds Associated to a Given One\/}. In preparation.

\bibitem{baybrun} D.~Bayen and  R.~Bruner,
{\em Real Connective $K$-Theory and the Quaternion Group\/}, Trans. AMS, {\bf 348}, 
(1996): 2201--2216.

\bibitem{bousfield.kan}A.~K. Bousfield and D.~M. Kan. \emph{Homotopy Limits, Completions
and Localizations}, Lecture notes in Mathematics {\bf 304}, Springer Verlag, 1972.

\bibitem{buchstaber.panov.2} V.~Buchstaber and T.~Panov,{\em  Torus Actions and their 
Applications in Topology and Combinatorics }, AMS University Lecture Series, volume {\bf 24}, 
(2002).

\bibitem{buchstaber.ray} V.~Buchstaber and N.~Ray, {\em An Invitation to Toric Topology:
vertex four of a remarkable tetrahedron\/} . Proceedings of the International Conference in 
Toric Topology, Osaka 2006, edited by Megumi Harada, Yael Karshon, Mikiya Masuda, and 
Taras Panov. AMS Contemporary Mathematics {\bf 460}:1--27, 2008.

\bibitem{civan.ray} Y.~Civan and N.~Ray, {\em Homotopy Decompositions and Real K-Theory 
of Bott Towers}, $K$-Theory {\bf 34} (2005), 1--33.

\bibitem{danilov}V.~I.~ Danilov, {\em The Geometry of Toric Varieties}, Russian Math. Surveys 
{\bf 33}. (1978), 97--154; Uspekhi Mat. Nauk 33 (1978), 85--134.

\bibitem{davis.jan} M.~Davis, and T.~Januszkiewicz, {\em Convex polytopes,
Coxeter orbifolds and torus actions}, Duke Math. J. {\bf 62} (1991), no.
2, 417--451.

\bibitem{denham.suciu} G.~Denham and A.~Suciu, {\em Moment-angle complexes,
monomial ideals and Massey products}, arXiv:math/0512497v4
[math.AT].

\bibitem{dobson} A.~Dobson, {\em $KO$-theory of Thom complexes}, University of
Manchester thesis, 2007.

\bibitem{mahowald.ray} M.~Mahowald and N.~Ray, {\em A Note on the Thom Isomorphism},
Proceedings of the American Mathematical Society, {\bf 82\/} (2):307--308, 1981.

\bibitem{nishimura} Y.~Nishimura, {\em The Quasi $KO$-types of Toric Manifolds}, 
In Proceedings of the International Conference in Toric Topology; Osaka City University 
2006, 287--292. Edited by Megumi Harada, Yael Karshon, Mikiya Masuda, and Taras 
Panov. Contemporary Mathematics {\bf 460\/}. AMS, Providence RI, 2008.

\bibitem{notbohm.ray} D.~Notbohm and N.~Ray. \emph{On Davis-Januszkiewicz
Homotopy Types~I; formality and rationalization}, Algebraic and
Geometric Topology {\bf 5}:31--51, ( 2005).


\bibitem{panov.ray.vogt} T.~Panov, N.~ Ray, and R.~Vogt, {\em  Colimits, Stanley-
Reisner Algebras, and Loop Spaces}, In Algebraic Topology: Categorical Decomposition 
Techniques, Progress in Mathematics {\bf 215}:261--291. Birkh\"{a}user, Basel, 2003.

\

\

\end{thebibliography}

\end{document}